\documentclass[a4paper]{amsart}

\usepackage[british,english]{babel} 
\usepackage{graphics,color,pgf}
\usepackage{epsfig}
\usepackage[ansinew]{inputenc}
\usepackage[all]{xy}
\usepackage{hyperref}
\usepackage{amssymb, amsmath,amsthm, mathtools, amscd}
\usepackage{mathrsfs}
\usepackage{stmaryrd}
\usepackage{cancel} 
\usepackage{tikz}
\usetikzlibrary{matrix}

\theoremstyle{plain}
 \newtheorem{thm}{Theorem}
 \newtheorem{prop}[thm]{Proposition}
 \newtheorem{lem}[thm]{Lemma}
 \newtheorem{cor}[thm]{Corollary}
 \newtheorem*{op*}{Open Problem} 
 \newtheorem*{frthmA*}{Theorem A} 
 \newtheorem*{frcorB*}{Corollary B} 
  \newtheorem*{frthmAp*}{Theorem A'} 
   
\theoremstyle{definition}
 
 \newtheorem{dfn}[thm]{Definition}
  \newtheorem{conj}[thm]{Conjecture}
\theoremstyle{remark}
 \newtheorem*{rem}{Remark}

\newcommand{\R}{\mathbb{R}}

\DeclareMathOperator{\id}{\mathrm{id}}
\DeclareMathOperator{\C}{\mathrm{C}} 

\renewcommand{\leq}{\leqslant}
\renewcommand{\geq}{\geqslant}
\renewcommand{\setminus}{\smallsetminus}
\DeclareMathOperator\bbR{\mathbb{R}}

\title[Some explicit cocycles on the boundary]{Some explicit cocycles on the Furstenberg boundary for products of isometries of hyperbolic spaces and $\mathrm{SL}(3,\mathbb{K})$}

\author[Michelle Bucher]{
Michelle Bucher} 
\address{Universit\'e de Gen\`eve}
\email{Michelle.Bucher-Karlsson@unige.ch}

\author[Alessio Savini]{
Alessio Savini} 

\address{ 
University of Milano-Bicocca
}
\email{alessio.savini@unimib.it}

\thanks{Supported by the Swiss National Science Foundation} 

\begin{document}

\begin{abstract}
In \cite{Monod}, Nicolas Monod showed that the evaluation map $H^*_m(G\curvearrowright G/P)\longrightarrow H^*_m(G)$ between the measurable cohomology of the action of a connected semisimple Lie group $G$ on its Furstenberg boundary $G/P$ and the measurable cohomology of $G$ is surjective with a kernel that can be entirely described in terms of invariants in the cohomology of a maximal split torus $A<G$. In \cite{BuSaAlt} we refine Monod's result and show in particular that the cohomology of non-alternating cocycles on $G/P$ is in general not trivial and lies in the kernel of the evaluation. In this paper we describe explicitly such non-alternating and alternating cocycles on $G/P$ in low degrees when $G$ is either a product of isometries of real hyperbolic spaces or $G=\mathrm{SL}(3,\mathbb{K})$, where $\mathbb{K}$ is either the real or the complex field. As a consequence, we deduce that the comparison map $H^*_{m,b}(G)\rightarrow H^*_m(G)$ from the measurable bounded cohomology is injective in degree $3$ for nontrivial products of isometries of hyperbolic spaces. We get also another proof of the injectivity for $G=\mathrm{SL}(3,\mathbb{K})$, when $\mathbb{K}$ is either the real field or the complex one. \end{abstract}

\maketitle

\section{Introduction}

Let $G$ be a connected semisimple Lie group with finite center. Its measurable cohomology $ H^*_m(G)$ is defined as the cohomology of the complex
$$L^0(G^{*+1},\mathbb{R})^G$$
endowed with its usual homogeneous differential. For any homogeneous $G$-space  $X$ one can similarly consider the measurable cohomology of the action $H^*_m(G\curvearrowright X)$ as the cohomology of the complex
$$L^0(X^{*+1},\mathbb{R})^G$$
also endowed with its homogeneous differential. Evaluation on a generic base point induces a map
\begin{equation}\label{eval map}
ev:H^*_m(G\curvearrowright X)\longrightarrow H^*_m(G)\end{equation}
which is easily shown not to depend on the base point. By Austin and Moore \cite[Theorem A]{AM} the measurable cohomology of $G$ is well understood: it is isomorphic to its continuous cohomology. The latter can be identified by the Van Est isomorphism with the singular cohomology of its compact dual symmetric space. However, in an aim to obtain explicit cocycles representing classes in $H^*_m(G)$, it is desirable to find the nicest (in terms of dimension, topological properties, etc) homogeneous $G$-space $X$ for which the evaluation map (\ref{eval map}) is surjective. For example, it is easy to see that if $K<G$ is a maximal compact subgroup, the evaluation map (\ref{eval map}) is in fact an isomorphism for the symmetric space $X=G/K$. 

Let now $P<G$ be a minimal parabolic subgroup. Monod remarkably showed that the evaluation map is surjective for the action on the Furstenberg boundary $X=G/P$. Furthermore, he gives a precise description of the kernel in terms of the $w_0$-invariant cohomology $H^*_m(A)^{w_0}$ of a maximal split torus $A<P$, where $w_0$ is a representative of the longest element of the Weyl group associated to $A$. Recall that the $w_0$-invariant cohomology of $A$ is isomorphic to the exterior algebra $(\wedge\mathfrak{a}^*)^{w_0}$, where the action of $w_0$ on $\wedge\mathfrak{a}^*$ is induced by the adjoint representation.

\begin{thm}\cite[Theorem B]{Monod}\label{Thm Monod} Let $G$ be a connected semisimple Lie group with finite center. The evaluation map 
$$ H^*_m(G\curvearrowright G/P)\longrightarrow H^*_m(G)$$
is surjective and its kernel 
$$NH^*_m(G\curvearrowright G/P):= \mathrm{Ker}(H^*_m(G\curvearrowright G/P)\longrightarrow H^*_m(G))$$
fits into an exact sequence
\begin{equation}\label{SES}(\wedge^{k-2}\mathfrak{a}^*)^{w_0} \overset{i}\longrightarrow NH^k_m(G\curvearrowright G/P) \longrightarrow  (\wedge^{k-1}\mathfrak{a}^*)^{w_0}\end{equation}
when $k\geq 3$, and for $k =2$, we have an isomorphism 
$$
NH^2_m(G \curvearrowright G/P) \cong (\mathfrak{a}^\ast)^{w_0}. 
$$\end{thm}

The usual alternation map on $L^0((G/P)^{q+1})$ is idempotent and it induces a splitting into \emph{non-alternating} and \emph{alternating} cochains. Since alternation is a cochain map commuting with the $G$-action, we obtain a decompositon of the cohomomology on the boundary $G/P$ as follows:
$$
H^*_m(G \curvearrowright G/P) \cong H^\ast_{m, \mathrm{n-alt}}(G \curvearrowright G/P) \oplus H^\ast_{m,\mathrm{alt}}(G \curvearrowright G/P),
$$
and the same holds for the kernel:
$$
NH^\ast_m(G \curvearrowright G/P) \cong NH^\ast_{m,\mathrm{n-alt}}(G \curvearrowright G/P) \oplus NH^\ast_{m,\mathrm{alt}}(G \curvearrowright G/P). 
$$
In \cite[Theorem 2, Theorem 3]{BuSaAlt} we refine Theorem \ref{Thm Monod} and show in particular that
$$(\wedge^{k-2}\mathfrak{a}^*)^{w_0}\cong H^k_{m,\mathrm{n-alt}}(G\curvearrowright G/P) \cong NH^k_{m,\mathrm{n-alt}}(G \curvearrowright G/P),$$
$$(\wedge^{k-1}\mathfrak{a}^\ast)^{w_0} \cong NH^k_{m,\mathrm{alt}}(G \curvearrowright G/P).$$
As a consequence the short exact sequence of Equation \eqref{SES} corresponds exactly to the decomposition into non-alternating and alternating cochains on the boundary. 

In this paper, we propose to study the kernel in more details following our belief that explicit cohomology classes should be represented by explicit cocycles. We will thus describe the inclusion $i$ and produce a section $s$ of the short exact sequence \eqref{SES} in low degrees for some families of groups, leading to very explicit $G$-invariant cocycles on the boundary $G/P$. In fact, it is the discovery of such explicit non-alternating cocycles which motivated our paper \cite{BuSaAlt}. 


\subsection{Products of isometries of real hyperbolic spaces} Consider first the case of one factor $G=\mathrm{Isom}^+(\mathbb{H}^n)$. If we look at the upper half space model, we identify $\partial \mathbb{H}^n$ with $\mathbb{R}^{n-1}\cup \{{\infty} \}$, we fix the stabilizer $P=\mathrm{Stab}({\infty})$ and take for maximal abelian subgroup $A<P$ the group of homotheties, namely $A:=\{a_\lambda:x\mapsto \lambda x\mid \lambda\in \mathbb{R}_{>0}\}$.

Given a $4$-tuple of distinct points $x_0,\dots,x_3\in \mathbb{R}^{n-1}\cup \{{\infty} \}=\partial \mathbb{H}^n$ we define their \emph{(positive) cross ratio}  by 
\begin{equation}\label{eq:cross:ratio:real}
{b}(x_0,\cdots,x_3)=\frac{\|x_2 -x_0\| \| x_3 - x_1\|}{\|x_2-x_1\|\|x_3-x_0\|}\in \mathbb{R}_{>0}, 
\end{equation}
where $\| \cdot \|$ denotes the Euclidean norm and we assume that ${\infty}/{\infty}=1$. In the case of $\partial \mathbb{H}^2=P^1\mathbb{R}$, respectively $\partial \mathbb{H}^3=P^1\mathbb{C}$, we get the absolute value, respectively the modulus, of the classic cross ratio of 4 points. It is not hard to verify that the positive cross ratio is  invariant under the isometry group of $\mathbb{H}^n$. 

Suppose now that $G=G_1\times \dots \times G_k$, where $G_i=\mathrm{Isom}^+(\mathbb{H}^{n_i})$ for some $n_i\geq 2$. Take $P=P_1\times \dots \times P_k$ and $A=A_1\times \dots \times A_k\cong \mathbb{R}^k$, with $A_i<P_i<G_i$ as above for each factor. We denote by $b_i$ the cross ratio defined on each of these factors.


We will identify later a representative of the longest element $w_0\in G$, but even before doing so, it is clear that it has to act as $-1$ on $\mathfrak{a}$ since it acts non trivially and hence by $-1$ on each factor $\mathfrak{a}_i$. As a consequence, for $\ell\geq 3$, Monod's Theorem translates into 
$$NH_m^\ell(G\curvearrowright G/P)\cong \left\{ \begin{array}{ll}
(\wedge^{\ell-1}\mathfrak{a})^* &\mathrm{if \ }   \ell \mathrm{\ is \ odd,}\\
(\wedge^{\ell-2}\mathfrak{a})^* & \mathrm{if \ }   \ell \mathrm{\ is \ even.}
\end{array}\right.$$
It should be clear that the above kernel vanishes when $\ell$ is greater than $k+2$.  We give explicit versions of these isomorphisms at cocycle level in degree $3$ and $4$: 

\begin{thm}\label{Theorem 2 to 3 for prod of Hn} A section 
$$s:(\wedge^{2}\mathfrak{a})^*\longrightarrow NH^3_m(G\curvearrowright \Pi \partial \mathbb{H}^{n_i})$$ 
of the short exact sequence (\ref{SES}) is given by sending an alternating form $\alpha_\mathfrak{a}\in (\wedge^{2}\mathfrak{a})^*$ to the measurable almost everywhere defined $G$-invariant alternating cocycle 
$$\begin{array}{rcl}
s(\alpha_{\mathfrak{a}}):  ( \Pi \partial \mathbb{H}^{n_i})^4&\longrightarrow &\mathbb{R}\\
(x_0,\dots,x_3)&\longmapsto &4\alpha_\mathfrak{a}(\log (a_{b_i(x_0,x_1,x_2,x_3)})_{i=1}^k,\log (a_{b_i(x_1,x_2,x_3,x_0)})_{i=1}^k), 
\end{array}$$
where $\log$ denotes the logarithm sending $A$ to its Lie algebra $\mathfrak{a}$. 
\end{thm}


For $G=\mathrm{PSL}(2,\mathbb{R}) \times \mathrm{PSL}(2,\mathbb{R})$ this cocycle (on tuples which are positively oriented in each factor) is explicitly given in \cite{Monod}.  All other cases are new. The case of several $2$-dimensional factors is a straightforward generalization of Monod's cocycle, but the higher dimensional case, relying on explicitly exploiting the $3$-transitivity of the action on the boundary, is more involved as $K\cap P$ is not trivial in dimension $n\geq 3$, where $K=\mathrm{SO}(n)$ is the maximal compact subgroup of $\mathrm{Isom}^+(\mathbb{H}^n)$.  

We will prove Theorem \ref{Theorem 2 to 3 for prod of Hn} in Section \ref{SS for prod of real hyp} by following the spectral sequence (that we recall in Section \ref{Spectral sequence})  introduced by Monod  to establish Theorem  \ref{Thm Monod}. Note however that in this case it is easy to show retrospectively that  $NH^3_m(G\curvearrowright \Pi \partial \mathbb{H}^{n_i})$ is generated by the classes given in Theorem \ref{Theorem 2 to 3 for prod of Hn}.  Indeed, any such cochain is easily verified to be an alternating cocycle which takes infinitely many and unbounded values. As a consequence, it cannot be the coboundary of a $2$-cochain, as those are constant functions by $3$-transitivity. Furthermore, it is trivially in the kernel of the evaluation map since $H^3_m(G)=0$. Indeed $H^*_m(G)$  is generated (as an algebra) by the cup products of an even number of volume classes. Finally, a similar coboundary argument combined with our Proposition \ref{Omega3 unbounded hyperbolic} shows that there are no relations among these classes so that they generate a subspace of dimension $\binom{k}{2}$ in the  $\binom{k}{2}$-dimensional (by Theorem \ref{Thm Monod}) space $NH^3_m(G\curvearrowright \Pi \partial \mathbb{H}^{n_i})$. The point of our proof is thus merely, beside showing that the natural isomorphism induced by the spectral sequence really is given by the map in Theorem \ref{Theorem 2 to 3 for prod of Hn}, to show how to find these explicit cocycles. 

The classes in degree $4$ are given by a surprisingly similar expression, which is all the more astonishing as the intermediate expressions we obtain following the spectral sequence for Theorems \ref{Theorem 2 to 3 for prod of Hn} and \ref{Theorem 2 to 4 for prod of Hn} show no similarity until their very last expression. 

\begin{thm}\label{Theorem 2 to 4 for prod of Hn}
The isomorphism
$$i:(\wedge^{2}\mathfrak{a})^*\longrightarrow NH^4_m(G\curvearrowright \Pi \partial \mathbb{H}^{n_i})$$ 
of the short exact sequence (\ref{SES}) is given by sending an alternating form $\alpha_\mathfrak{a}\in (\wedge^{2}\mathfrak{a})^*$ to the measurable almost everywhere defined $G$-invariant non-alternating cocycle
$$\begin{array}{rcl}
i(\alpha_{\mathfrak{a}}): ( \Pi \partial \mathbb{H}^{n_i})^5&\longrightarrow &\mathbb{R}\\
(x_0,\dots,x_4)&\longmapsto &\alpha_\mathfrak{a}(\log (a_{b_i(x_0,x_1,x_2,x_3)})_{i=1}^k,\log (a_{b_i(x_1,x_2,x_3,x_4)})_{i=1}^k),
\end{array}$$
where $\log$ denotes the logarithm sending $A$ to its Lie algebra $\mathfrak{a}$. \end{thm}

To see that these cocycles are non-alternating, observe that $i(\alpha_{\mathfrak{a}})(x_0,\dots,x_4)=-i(\alpha_{\mathfrak{a}})(x_4,\dots,x_0)$. Since there is a factor $-1$, whereas the sign of the permutation $(x_0,\dots,x_4)\mapsto (x_4,\dots,x_0)$ is $+1$, the alternation of these cocycles automatically vanishes. In the case of two factors, the computation is carried out in \cite[Proposition 18]{BuSaAlt}.

While it is easy to show that these cochains are indeed invariant cocycles on $ \Pi \partial \mathbb{H}^{n_i}$, we do not know, in contrast to Theorem \ref{Theorem 2 to 3 for prod of Hn}, how to even show their nontriviality in cohomology without going through the spectral sequence. The proof of Theorem \ref{Theorem 2 to 4 for prod of Hn} will be presented in Section \ref{SS for prod of real hyp}.


\begin{rem} Our proof also gives an explicit expression in higher degrees for the inclusion and the section. However, the expression is a sum of $(\ell-1)3^{\ell-1}$, respectively $3^{\ell-2}$, evaluations of $\alpha_\mathfrak{a}$ on various cross ratios in degree $\ell$ for $\ell$ odd, respectively even, which we were unable to simplify. \end{rem}

\subsection{The special linear group $\mathrm{SL}(3,\mathbb{K})$ for $\mathbb{K}=\mathbb{R},\mathbb{C}$}

We give a complete description of the kernel $NH^\ast_m(G \curvearrowright G/P)$ for the first cases of irreducible Lie groups of higher rank, namely $G=\mathrm{SL}(3,\mathbb{K})$, for $\mathbb{K}=\mathbb{R},\mathbb{C}$. In this context we fix $P$ as the group of upper triangular matrices with $\mathbb{K}$-entries and we consider $A<P$ the subgroup of diagonal matrices with positive real entries. This time, the adjoint action of the longest element $w_0$ is given on the $2$-dimensional Lie algebra $\mathfrak{a}$ by a reflection, so there is an eigenspace of fixed vectors and another one associated to the eigenvalue $-1$. As a consequence, the $w_0$-invariant cohomology of $A$ is not trivial only in degree one  (Lemma \ref{lem invariant SL3}) and it determines the kernel $NH^\ast_m(G \curvearrowright G/P)$ in degree $2$ and $3$. 

In this setting the boundary $G/P=\mathcal{FL}(3,\mathbb{K})$ parametrizes \emph{complete flags} in $P^2(\mathbb{K})$, namely pairs $(p,\ell)$, where $p$ is a point on the line $\ell$. The quotient of the space of tuples of flags with respect to the diagonal action of $G$ is called \emph{configuration space} and it has been widely studied so far. In the particular case of triples and $4$-tuples, several authors contributed to give an explicit parametrization of a subset of full measure, namely the one of flags in \emph{general position} (see the works by Bergeron-Falbel-Guilloux \cite{BFG14} or  Dimofte-Gabella-Goncharov \cite{DGG16}). We refer to Section \ref{sec general position} for a precise definition. 

A complete invariant for the class of a triple $(F_0,F_1,F_2)$ of flags in general position is the \emph{triple ratio} $\tau(F_0,F_1,F_2)\in \mathbb{K}\setminus\{0,-1\}$. Despite its name, the triple ratio is a cross ratio computed on the line $\ell_1$ in terms of the intersection scheme of the given triple (see Definition \ref{dfn triple ratio}).


\begin{thm}\label{Thm 1 to 2 for sl3}

 A section 
$$s:(\mathfrak{a}^*)^{w_0}\cong \mathbb{R} \longrightarrow NH^2_m(G\curvearrowright \mathcal{FL}(3,\mathbb{K}))$$ 
of the short exact sequence (\ref{SES}) is given by sending a $w_0$-invariant linear form $\alpha_{\mathfrak{a}} \in (\mathfrak{a}^\star)^{w_0}$ to the measurable almost everywhere defined alternating $G$-invariant cocycle 
$$\begin{array}{rcl}
s(\alpha): (\mathcal{FL}(3,\mathbb{K}))^3 &\longrightarrow &\mathbb{R}\\
(F_0,F_1,F_2) & \longmapsto & \alpha_{\mathfrak{a}}(\log |\tau(F_0,F_1,F_2)|).
\end{array}$$
\end{thm}

\begin{rem} An explicit isomorphism $(\mathfrak{a}^*)^{w_0}\cong \mathbb{R}$ is given  in Subsection \ref{SL3R2}  and is used in  the proof of Theorem \ref{Thm 1 to 2 for sl3}. This remark applies also to Theorem \ref{Thm 1 to 3 for sl3}.
\end{rem}

\begin{rem}
Thanks to personal communications of Elisha Falbel with the first author, we understood that the cocycle realizing the explicit expression of the section can be written as a coboundary when viewed on the space of \emph{affine} flags. If we write an affine flag as $F:=(v,f) \in \mathcal{FL}_\mathrm{aff}(3,\mathbb{K})$, where $v$ is a non-zero vector of $\mathbb{K}^3$ and $f$ is linear functional vanishing on $v$, for any $\alpha \in (\mathfrak{a}^\ast)^{w_0}$ the function 
$$
h_\alpha:(\mathcal{FL}_{\mathrm{aff}}(3,\mathbb{K}))^2 \rightarrow \mathbb{R}, \ \ h_\alpha(F_0,F_1):=\alpha \left(\log\left| \frac{f_0(v_1)}{f_1(v_0)} \right|\right)
$$
is measurable and satisfies $dh_\alpha=s(\alpha)$ by \cite[Section 3.3]{FW17}. 
\end{rem}

The proof of Theorem \ref{Thm 1 to 2 for sl3} can be found in Section \ref{sec proof 1 2 sl3} and it relies on a careful study of Monod's spectral sequence. As is the case for Theorem \ref{Theorem 2 to 3 for prod of Hn}, a simpler proof is possible to show that $NH^2_m(G\curvearrowright \mathcal{FL}(3,\mathbb{K}))$ is generated by one such nonzero cocycle. Indeed, it is immediate to see that the triple ratio is a multiplicative cocycle and hence its logarithm is a cocycle lying in $L^0((G/P)^3)^{G}$ which further trivially is in the kernel of the evaluation map, since $H_m^2(\mathrm{SL}(3,\mathbb{K}))$ vanishes for $\mathbb{K}=\mathbb{R},\mathbb{C}$. Alternatively, by the second remark after Theorem \ref{Thm 1 to 3 for sl3}, it is easy to exhibit an explicit coboundary in $L^0(G^2)^G$. Additionally, being non-constant, it determines a non-trivial cohomology class in $H^2_m(G \curvearrowright \mathcal{FL}(3,\mathbb{K}))$. Indeed, $G$ acts transitively on pairs of flags in general position, so coboundaries must be constant. 

The case $NH^3_m(G \curvearrowright \mathcal{FL}(3,\mathbb{K}))$ can be treated similarly. This time configurations of $4$-tuples of flags in general position are parametrized by four \emph{cross ratios} $(z_{01},z_{10},z_{23},z_{32})$. For each $z_{ij}$ the indices $(i,j)$ refer to the line on which the cross ratio is computed and to the intersection scheme one has to consider. For a precise definition we refer the reader to Definition \ref{dfn cross ratios tuple}.

\begin{thm}\label{Thm 1 to 3 for sl3}

 The isomorphism
$$i:(\mathfrak{a}^*)^{w_0}\longrightarrow NH^3_m(G\curvearrowright \mathcal{FL}(3,\mathbb{K}))$$ 
of the short exact sequence (\ref{SES}) is given by sending a $w_0$-invariant linear form $\alpha_\mathfrak{a}\in (\mathfrak{a}^*)^{w_0}$ to the measurable almost everywhere defined $G$-invariant non-alternating cocycle 
$$\begin{array}{rcl}
i(\alpha_{\mathfrak{a}}): (\mathcal{FL}(3,\mathbb{K}))^4 &\longrightarrow &\mathbb{R}\\
(F_0,\dots,F_3) & \longmapsto & -\frac{2}{3}\alpha_{\mathfrak{a}}(\log(|z_{10}|/|z_{23}|)),
\end{array}$$
where $\log$ denotes the logarithm sending $A$ to its Lie algebra $\mathfrak{a}$.
\end{thm}

\begin{rem}
Also in this case, for any $\alpha \in (\mathfrak{a}^\ast)^{w_0}$, we can check that $i(\alpha)$ is a coboundary when viewed on the space of affine flags $\mathcal{FL}_{\mathrm{aff}}(3,\mathbb{K})$. If we define
$$
h_\alpha:(\mathcal{FL}_{\mathrm{aff}}(3,\mathbb{K}))^3 \rightarrow \mathbb{R}, \ \ h_\alpha(F_0,F_1,F_2):=-\frac{2}{3}\alpha \left(\mathrm{\log}\left|\frac{f_1(v_0)f_1(v_2)}{\mathrm{det}(v_0,v_1,v_2)}\right| \right),
$$
then again $dh_\alpha=i(\alpha)$ by \cite[Equation 3.4.1]{FW17}. 
\end{rem}

The fact that these cocycles are non-alternating follows as for the cocycles in degree $4$ on $ \Pi \partial \mathbb{H}^{n_i}$ from the fact that $i(\alpha_{\mathfrak{a}})(F_0,\dots,F_3)=-i(\alpha_{\mathfrak{a}})(F_3,\dots,F_0)$, which is easily deducible from the definition of the cross ratios of flags. 

Again, our proof is based on Monod'  spectral sequence. But in this case we rely on the open source software \texttt{SageMath} to perform some computations. 

Here also a more direct approach to show that $NH^3_m(G\curvearrowright \mathcal{FL}(3,\mathbb{R}))$ is generated by one such nonzero cocycle is possible: First, it is not hard to verify by hand that $\log(|z_{10}|/|z_{23}|)$ is a cocycle lying in $L^0((G/P)^4)^{G}$. Second, an explicit coboundary in $L^0(G^3)^G$ is easy to produce (see the remark after Theorem \ref{Thm 1 to 3 for sl3}) so that these cocycles lie in the kernel of the evaluation maps.  Third, a computation based on the relations between cross ratios and triple ratios of a $4$-tuple allows to show that the cocycle represents a nontrivial class in $NH^3_m(G\curvearrowright \mathcal{FL}(3,\mathbb{R}))$.



\subsection{(Un)bounded cohomology}

The \emph{bounded (measurable) cohomology} of a group or an action, denoted by $H^*_{m,b}(G)$ and $H^*_{m,b}(G\curvearrowright X)$, respectively, is defined by considering the complexes of $L^{\infty}$ functions instead of $L^0$. The inclusion $L^{\infty}\hookrightarrow L^0$ naturally induces comparison maps between the bounded and unbounded measurable cohomology groups, so that we obtain a commutative diagram
\begin{equation*}
\xymatrix{
H^*_{m,b}(G\curvearrowright X) \ar[r]\ar[d]& H^*_{m,b}(G)\ar[d]\\
H^*_{m}(G\curvearrowright X) \ar[r] & H^*_m(G).}
\end{equation*}
In the case $X=G/P$, by the amenability of $P$, the upper horizontal map in the above commutative diagram is an isomorphism \cite[Corollary 3.8]{BuModeg2}. It is still a mysterious open conjecture that the vertical right arrow should be an isomorphism \cite[Problem A]{MonodICM}. For the surjectivity, Dupont \cite[Remark 3]{Dupont} more generally asked whether a particular family of cocycles obtained by integration of straightened simplices in the symmetric space $G/K$ are bounded, which would imply surjectivity of this map. We will concentrate on the injectivity:

\begin{conj} \label{conj Monod} \cite[Problem A]{MonodICM} Let $G$ be a connected semisimple Lie group with finite center. The comparison map 
$$H^*_{m,b}(G)\longrightarrow H^*_m(G)$$
is injective. 
\end{conj}
Conjecture \ref{conj Monod} is trivial in degrees $0$ and $1$. It is further known to hold: 

\noindent In degree $2$: in all generality \cite{BuModeg2}.

\noindent In degree $3$: for $\mathrm{SL}(2,\mathbb{R})$ \cite{BuModeg3}, for $\mathrm{SL}(2,\mathbb{C})$ \cite{Bloch}, for $\mathrm{SL}(n,\mathbb{R})$ \cite{MonodStab}, for $\mathrm{SL}(n,\mathbb{C})$ \cite{MonodStab,BBI18}, for $\mathrm{Isom}^+(\mathbb{H}^n)$ \cite{Pieters, Monod}, for $\mathrm{Sp}_{2k}(\mathbb{C})$ \cite[Theorem F]{DeLaCruz}. 

\noindent In degree $4$: remarkably only for $\mathrm{SL}(2,\mathbb{R})$, which is an achievement by Hartnick and Ott \cite{HaOtt}.

Given Monod's description of the measurable cohomology on the boundary, it is natural to decompose the above conjecture into the two following weaker conjectures: 

\begin{conj} \label{Conj injectivity boundary} The comparison map 
$$H^*_{m,b}(G\curvearrowright G/P)\longrightarrow H^*_m(G\curvearrowright G/P)$$
is injective. 
\end{conj}

\begin{conj} \label{Conj Gap}The intersection of $NH^*_m(G\curvearrowright G/P)$ with the image of the comparison map 
$$H^*_{m,b}(G\curvearrowright G/P)\longrightarrow H^*_m(G\curvearrowright G/P)$$
is $0$. 
\end{conj}

\begin{lem} \label{lem conj equals two half conj} Conjecture \ref{conj Monod} is equivalent to the validity of both Conjectures \ref{Conj injectivity boundary} and   \ref{Conj Gap}. 
\end{lem}

\begin{proof} This is straightforward from the commutative diagram
\begin{equation*}
\xymatrix{
&H^*_{m,b}(G\curvearrowright G/P) \ar[r]^{\hspace{10pt} ev_b} \ar[d]^{c_{G/P}}& H^*_{m,b}(G)\ar[d]^{c_G}\\
NH^*_{m}(G\curvearrowright G/P) \ar[r]&H^*_{m}(G\curvearrowright G/P) \ar[r]^{\hspace{10pt} ev} & H^*_m(G).}
\end{equation*}
Indeed the comparison map $c_G$ is injective if and only if its precomposition with the isomorphism $ev_b$ is injective, but by the commutativity of the diagram this is
$$c_G\circ ev_b=ev\circ c_{G/P}.$$
Now, injectivity of the latter composition is equivalent to 
\begin{itemize}
\item the injectivity of $c_{G/P}$, which is precisely the statement in Conjecture \ref{Conj injectivity boundary} and 
\item the injectivity of $ev$ on the image of $c_{G/P}$, which is equivalent to requiring that the image of $c_{G/P}$ intersects the kernel of $ev$ trivially, i.e. Conjecture  \ref{Conj Gap}. \end{itemize}\end{proof} 

We follow the strategy of decomposing Conjecture  \ref{conj Monod} into Conjectures \ref{Conj injectivity boundary} and   \ref{Conj Gap} to prove Monod's conjecture in degree $3$ for the groups we have considered:

\begin{thm} \label{injectivity for products} Let $G=\prod_{i=1}^k \mathrm{Isom}^+(\mathbb{H}^{n_i})$, where $n_i \geq 2$ for all $1 \leq i \leq k$. The map
 $$c_{G}:H^3_{m,b}(G)\longrightarrow H_m^3(G)$$
 is injective. 
\end{thm}

Observe that in this case, surjectivity obviously holds.  Theorem \ref{injectivity for products}  was known in the case of one factor $G=\mathrm{Isom}^+(\mathbb{H}^n)$ \cite{Bloch,Pieters,Monod}, but is new for more than one factor. Note also that in contrast to measurable cohomology, no K\"unneth formula is known for bounded cohomology, so the injectivity cannot be deduced from the factor case. 

Using the same strategy, we obtain an alternative proof of the following known result: 
\begin{thm} \label{Injectivity comparison sl3} Let $\mathbb{K}$ be the real or complex field. The comparison map 
$$c_G:H^3_{m,b}(\mathrm{SL}(3,\mathbb{K}))\longrightarrow H^3_m(\mathrm{SL}(3,\mathbb{K}))$$
is injective. 
\end{thm}

When $\mathbb{K}=\mathbb{R}$, the result is due to Monod \cite[Theorem 1.2]{MonodStab}. As it is well known that $H_m^3(\mathrm{SL}(3,\mathbb{R}))=0$, also in this case, the comparison map is an isomorphism.

When $\mathbb{K}=\mathbb{C}$, the result \cite[Theorem 2]{BBI18} was proved by the first author, Burger and Iozzi, using stability results from \cite{MonodStab}. In this case the group $H^3_m(\mathrm{SL}(3,\mathbb{C}))$ is generated by the Borel class, which is bounded. Thus the comparison map is surjective and hence again an isomorphism.

\subsection*{Plan of the paper.} We recall the spectral sequence used by Monod to prove Theorem \ref{Thm Monod} in Section \ref{Spectral sequence} and give explicit though theoretic descriptions of the nonzero differentials in Section \ref{Differentials}. We deal with the case of products of isometry groups of hyperbolic spaces in Section \ref{Section: Prod of hyp} by first recalling in Subsection \ref{Isom of Hn} the form of isometries in the upper half space model, where we also compute a little too many $\pi_A$-projections which will be all needed in the proofs of Theorems \ref{Theorem 2 to 3 for prod of Hn} and \ref{Theorem 2 to 4 for prod of Hn}. In Subsection \ref{homo for products} we exploit $3$-transitivity to describe the differentials in this case. Finally we prove Theorem   \ref{Theorem 2 to 3 for prod of Hn} in Subsection \ref{proof of theorem 2}, Theorem \ref{injectivity for products}  in Subsection \ref{proof of thm injectivitiy} and Theorem \ref{Theorem 2 to 4 for prod of Hn}  in Subsection \ref{proof of thm 3}. In our last Section \ref{SL3R}, we start by recalling some generalities like the triple ratio and cross ratio of triples or quadruples of flags in Subsection \ref{SL3R1}. We compute the dimensions of $H^*_m(A)^{w_0}=(\wedge\mathfrak{a}^*)^{w_0}$ in Subsection \ref{SL3R2} and give an expression for the $\pi_A$-projection in Subsection \ref{SL3R3}. Explicit differentials are exhibited in Subsection \ref{SL3R4}. Finally, Theorems \ref{Thm 1 to 2 for sl3} and \ref{Thm 1 to 3 for sl3} are proved in Subsections \ref{SL3R5} and \ref{SL3R6}, and the injectivity of the comparison map in degree $3$ for $\mathrm{SL}(3,\mathbb{K})$ (Theorem \ref{Injectivity comparison sl3}) is established in Subsection \ref{SL3R7}.

\subsection*{Acknowledgements} We are grateful to Elisha Falbel, Antonin Guilloux and Nicolas Monod  for several useful discussions in the preparation of this paper. 

\section{Monod's spectral sequence}\label{Spectral sequence}
Let $G$ be a connected semisimple Lie group with finite center, $P<G$ a minimal parabolic subgroup, $A<P$ a maximal split torus, $M=Z_K(\mathfrak{a})$ the centralizer of the Lie algebra $\mathfrak{a}$ in the maximal compact subgroup $K<G$ and $w_0$ a representative of the longest element of the Weyl group associated to $A$. The proof of Monod's Theorem \ref{Thm Monod} relies on the study of the spectral sequence associated to the double complex
$$C^{p,q}:=L^0(G^{p+1},L^0( (G/P)^q))^G,$$
endowed with two differentials 
$$
d^\uparrow: C^{p,q} \rightarrow C^{p+1,q}\ , \ d^\rightarrow: C^{p,q} \rightarrow C^{p,q+1}   .
$$
The vertical differential is the homogeneous differential on $G$, whereas the horizontal one is $(-1)^{p+1}$ times the homogenous differential on $G/P$ (to ensure that $d^\uparrow  d^\rightarrow=d^\rightarrow   d^\uparrow$). This bicomplex leads to two spectral sequences converging to the cohomology of the corresponding total complex. 

The spectral sequence of main interest starts at page $1$ with
\begin{equation}E_1^{p,q}=(H^p(C^{p,q}, d^\uparrow), d_1=d^\rightarrow).\label{E}\end{equation}
The other spectral sequence, obtained as above by exchanging the roles of $d^\uparrow$ and $d^\rightarrow$, is shown \cite[Proposition 6.1]{Monod} to collapse immediately to $0$ (from page $1$). We will see below what can be deduced from knowing that the spectral sequence from (\ref{E}) converges to $0$.

Recall that in general to find the image $[\alpha]\in E_k^{p,q}$ by the differential 
$$d_k:E_k^{p,q}\longrightarrow E_k^{p-k+1,q+k}$$
of a class represented by $\alpha\in C^{p,q}$, one needs to find a sequence of $\alpha_i\in C^{p-i,q+i}$, for $0\leq i \leq k-1$ such that
$$\alpha_0=\alpha \ \ \mathrm{and} \ \ d^\rightarrow \alpha_{i-1}=d^\uparrow \alpha_i.$$
The image $d_k([\alpha])$ is then simply represented by $d^\rightarrow \alpha_{k-1}$. 

\begin{figure}[!h]
\centering
\begin{tikzpicture}
  \matrix (m) [matrix of math nodes,
             nodes in empty cells,
             nodes={minimum width=9ex,
                    minimum height=9ex,
                    outer sep=-3pt},
             column sep=2ex, row sep=-1ex,
             text centered,anchor=center]{
         p      &         &          &          & \\
          \cdots & \cdots & \cdots & \cdots   & \cdots \\
          3    & \ H^3_m(G) \  & \ H^3_m(P) \ &\  H^3_m(A) \ & 0 & \ \cdots & \\
          2    &  \ H^2_m(G) \ &\  H^2_m(P) \  & \  H^2_m(A) \ & 0 & \  \cdots &  \\
          1    & \ H^1_m(G) \ & \ H^1_m(P) \ & \  H^1_m(A)\  & 0 & \   \cdots & \\
          0     & \ \mathbb{R}\   &\  \textup{L}^0(G/P)^G \ & \  \textup{L}^0((G/P)^2)^G  \ & \  \textup{L}^0((G/P)^3)^G\  &  \ \cdots & \\
    \quad\strut &   0  &  1  &  2  &  3  &  \cdots & q \strut \\};


\draw[->](m-3-2.east) -- (m-3-3.west)node[midway,above ]{$0$};
\draw[->](m-3-3.east) -- (m-3-4.west)node[midway,above ]{$d^\rightarrow$};
\draw[->](m-3-4.east) -- (m-3-5.west);

\draw[->](m-4-2.east) -- (m-4-3.west)node[midway,above ]{$0$};
\draw[->](m-4-3.east) -- (m-4-4.west)node[midway,above ]{$d^\rightarrow$};
\draw[->](m-4-4.east) -- (m-4-5.west);

\draw[->](m-5-2.east) -- (m-5-3.west)node[midway,above ]{$0$};
\draw[->](m-5-3.east) -- (m-5-4.west)node[midway,above ]{$d^\rightarrow$};
\draw[->](m-5-4.east) -- (m-5-5.west);

\draw[->](m-6-2.east) -- (m-6-3.west)node[midway,above ]{$\delta$};
\draw[->](m-6-3.east) -- (m-6-4.west)node[midway,above ]{$\delta$};
\draw[->](m-6-4.east) -- (m-6-5.west)node[midway,above ]{$\delta$};

\draw[thick] (m-1-1.east) -- (m-7-1.east) ;
\draw[thick] (m-7-1.north) -- (m-7-7.north) ;
\end{tikzpicture}
\caption{The first page $E_1$}\label{Page 1}
\end{figure}

\subsection*{First page of the spectral sequence} For the first column (for $q=0$), we obviously have $E_1^{p,0}=H^p_m(G)$. For the second and third columns, as $P$ and $MA$ are respectively the stabilizers of one and two (generic) point(s) on $G/P$, by the Eckmann-Shapiro isomorphism \cite[Theorem 6]{Moore} we have that
$E_1^{p,1}=H^p_m(P)$ and $E_1^{p,2}=H^p_m(MA) \cong H^p_m(A)$. The latter isomorphism holds since $M$ is compact and centralizes $A$. The cohomology of the remaining columns is trivial \cite[Proposition 5.1]{Monod} by the compactness of stablizers of a generic triple. In this way we are only left with the $G$-invariant of the coefficient modules $L((G/P)^q)^G$ in degree $0$, that is, for $E_1^{0,q}$. The first page is depicted in Figure \ref{Page 1}.

\subsection*{The differential $d_1$} It is straightforward to check that the differential 
$$d_1:E^{p,0}=H^p_m(G)\longrightarrow H^p_m(P)=E^{p,1}$$
induced by $d^\rightarrow$ is, as a map $H^p_m(G)\rightarrow H^p_m(P)$, simply $(-1)^{p+1}$ times the map induced by the inclusion $P<G$. What is  much less direct is the fact  that the restriction map is trivial. This is proved either in \cite[Corollary 3.2]{Monod} or \cite[Corollary 3]{Wienhard}.

 It is also clear that 
$$d_1:E^{0,q}=L^0((G/P)^q)^G\longrightarrow L^0((G/P)^{q+1})^G=E^{0,q+1}$$ 
translates into $(-1)$ times the homogeneous differential $\delta$. 

As for
$$d_1:E^{p,1}=H^p_m(P)\longrightarrow H^p_m(A)=E^{p,2},$$
it is shown in the proof of \cite[Proposition 4.1]{Monod}  that it is induced by 
\begin{equation}\label{d from P to A}\begin{array}{rccl}
d^\rightarrow:&C^{p,1}= L^0(G^{p+1})^P&\longrightarrow &L^0(G^{p+1})^A\\
 & f&\longmapsto &\{(g_0,\dots,g_p)\mapsto (-1)^{p+1}[f(w_0^{-1}g_0,\dots,w_0^{-1}g_p)-f(g_0,\dots,g_p)]\}.
\end{array}\end{equation}

The first differential, viewed as a map 
$$d_1:H^p_m(P)\cong H^p_m(A)\longrightarrow H^p_m(A)$$
is $(-1)^{p+1}2$ times the projection with kernel equal to the $w_0$-invariant cohomology classes in $H^*_m(A)$ and image the $w_0$-equivariant cohomology classes in $H^*_m(A)$ (for the obvious action by $\pm 1$ of $w_0$ on the coefficients $\mathbb{R}$). As a consequence the second page $E_2$ can be depicted as in Figure \ref{Page 2}

\begin{figure}[!h]
\centering
\begin{tikzpicture}
  \matrix (m) [matrix of math nodes,
             nodes in empty cells,
             nodes={minimum width=9ex,
                    minimum height=9ex,
                    outer sep=-3pt},
             column sep=2ex, row sep=-1ex,
             text centered,anchor=center]{
         p      &         &          &          & \\
          \cdots & \cdots & \cdots & \cdots   & \cdots \\
          3    & \ H^3_m(G) \  & \ H^3_m(A)^{w_0} \ &\  H^3_m(A)^{w_0} \ & 0 & \ \cdots & \\
          2    &  \ H^2_m(G) \ &\  H^2_m(A)^{w_0} \  & \  H^2_m(A)^{w_0} \ & 0 & \  \cdots &  \\
          1    & \ H^1_m(G) \ & \ H^1_m(A)^{w_0} \ & \  H^1_m(A)^{w_0} \  & 0 & \   \cdots & \\
          0     & \ 0\   &\  H^0_m(G \curvearrowright G/P) \ & \  H^1_m(G \curvearrowright G/P)  \ & \  H^2_m(G \curvearrowright G/P) \  &  \ \cdots & \\
    \quad\strut &   0  &  1  &  2  &  3  &  \cdots & q \strut \\};


\draw[->](m-3-2) -- (m-4-4);
\draw[->](m-3-3) -- (m-4-5);

\draw[->](m-4-2) -- (m-5-4);
\draw[->](m-4-3) -- (m-5-5);

\draw[->](m-5-2) -- (m-6-4);
\draw[->](m-5-3) -- (m-6-5);

\draw[thick] (m-1-1.east) -- (m-7-1.east) ;
\draw[thick] (m-7-1.north) -- (m-7-7.north) ;
\end{tikzpicture}
\caption{The second page $E_2$} \label{Page 2}
\end{figure}

\subsection*{The differential $d_2$} The differential 
$$d_2:E_2^{p,0}=H^{p}_m(G)\longrightarrow H^{p-1}_m(A)^{w_0}=E^{p-1,2}_2$$
vanishes for all $p\geq 1$. This is proved in \cite{Monod} only in the case where $w_0$ acts by $-1$ on $\mathfrak{a}$. We  therefore present the general case here. Let thus $\alpha_0\in C^{p,0}$ be representing a class $[\alpha]\in E_2^{p,0}=H^p_m(G)$. By construction, there exists $\alpha_1\in C^{p-1,1}$ such that $d^\uparrow \alpha_1=d^\rightarrow \alpha_0$ and
$$d_2([\alpha_0])=[d^\rightarrow \alpha_1]\in H^{p-1}_m(A)^{w_0}=E^{p-1,2}_2.$$
In particular, $d^\rightarrow \alpha_1$ represents a cohomology class in $H^{p-1}_m(A)=E^{p-1,2}_1$ which is $w_0$-equivariant. But since $d_1=d^\rightarrow:H^{p-1}_m(A)=E_1^{p-1,1}\rightarrow E_1^{p-1,2}=H^{p-1}_m(A)$ is $(-1)^{p}2$ times the projection on the $w_0$-equivariants of $H^{p-1}_m(A)$, the class 
$d^\rightarrow \alpha_1$ lies in its image and hence vanishes in $E^{p-1,2}_2$.

It follows that the only possibly nonzero differentials $d_2$, are the two differentials
$$d_2:E_2^{1,1}=H^1_m(A)^{w_0}\longrightarrow H^2_m(G \curvearrowright G/P)=E_2^{0,2}$$
and
$$d_2:E_2^{1,2}=H^1_m(A)^{w_0}\longrightarrow H^2_m(G \curvearrowright G/P)=E_2^{0,3}.$$ 

\subsection*{Nonzero differentials $d_p$} For $p\geq 2$, the only possibly nonzero differentials are, trivially, the three starting from the columns $0$, $1$ or $2$ with image on the first row.
 
 \subsection*{Conclusion} It follows that the only way that $E_2$ converges to $0$ is that
 $$d_{p-1}:H^{p-2}_m(A)^{w_0}\longrightarrow  H^{p}_m(G \curvearrowright G/P)$$
 is an injection, 
  $$d_{p}:H^{p-1}_m(A)^{w_0}\longrightarrow  H^{p}_m(G \curvearrowright G/P)/d_{p-1}(H^{p-2}_m(A)^{w_0})$$
 is an injection and
  $$d_{p+1}:H^{p}_m(G)\longrightarrow \big( H^{p}_m(G \curvearrowright G/P)/d_{p-1}(H^{p-2}_m(A)^{w_0})\big) / d_p(H^{p-1}_m(A)^{w_0}) $$
  is an isomorphism. 
  
 A standard diagram chase shows that the inverse of $d_3$ is indeed induced by the evaluation map, finishing the proof of Theorem \ref{Thm Monod}.

 \section{Explicit differentials} \label{Differentials}

The aim of this section is to exhibit explicit maps for the injection and a section of the short exact sequence (\ref{SES}) from Theorem \ref{Thm Monod}. We will thus present in Lemma \ref{Injection} and Proposition \ref{Section} maps at  cocycle level realizing the differentials 
$$d_{p-1}:H^{p-2}_m(A)^{w_0}\longrightarrow  H^{p}_m(G \curvearrowright G/P)$$
and
  $$d_{p}:H^{p-1}_m(A)^{w_0}\longrightarrow  H^{p}_m(G \curvearrowright G/P)/d_{p-1}((H^{p-2}_m(A))^{w_0}),$$
which in fact have images in $NH^{p}_m(G \curvearrowright G/P)$ and in the quotient $NH^{p}_m(G \curvearrowright G/P)/d_{p-1}((H^{p-2}_m(A))^{w_0})$, respectively. 

\subsection{Representing $H^p_m(A)^{w_0}$ in $C^{p,1}$ and $C^{p,2}$}

In Section \ref{Spectral sequence} we have established that
\begin{equation}\label{iso}H^p_m(A)^{w_0} \cong E_2^{p,1} \ \mathrm{and} \ H^p_m(A)^{w_0} \cong E_2^{p,2} .\end{equation}
For the explicit injection and section of the short exact sequence (\ref{SES}) presented below we will need to know that classes in $E_2^{p,1}$ and $E_2^{p,2}$ can be represented by cocycles in $C^{p,1}$ and $C^{p,2}$ with additional symmetries. Denote by $C^{p+1}(A)$ the cocomplex of continuous $A$-invariant cochains $A^{p+1}\rightarrow \mathbb{R}$ endowed with its homogenous boundary operator and let 
$$C^{p+1}(A)^{w_0}=\{\alpha \in C^{p+1}(A)\mid \alpha(w_0 a_0 w_0^{-1},\dots, w_0 a_p w_0^{-1})=\alpha(a_0,\dots,a_p)\}$$
the subspace of $w_0$-invariant cochains. 

The aim of this paragraph is to produce explicit maps
$$C^{p+1}(A)^{w_0} \rightarrow C^{p,1} \ \mathrm{and} \ C^{p+1}(A)^{w_0} \rightarrow C^{p,2}$$
inducing the isomorphisms (\ref{iso}).

We fix $\alpha \in C^{p+1}(A)$. Recall that in virtue of the Iwasawa decomposition, any element $g \in G$ can be written uniquely as $g=ank$, where $a \in A, n \in N, k \in K$. 
Here $N$ is the unipotent radical of the minimal parabolic subgroup and $K$ is a fixed maximal compact subgroup. The Iwasawa decomposition allows us to define a projection of $G$ onto $A$, namely
$$\begin{array}{rccl}
\pi_A:&G &\longrightarrow &A  \\
&g=ank&\longmapsto& a. 
\end{array}$$ 
Using such projection, we further define $\alpha_G=\pi_A^*(\alpha)\in L^0(G^{p+1})^P$ by
$$\alpha_G(g_0,\dots,g_p)=\alpha(\pi_A(g_0),\dots,\pi_A(g_p)).$$
Via the induction isomorphism $\alpha_G\in L^0(G^{p+1})^P\cong C^{p,1}\ni \overline{\alpha_G}$ we obtain the corresponding cocycle $\overline{\alpha_G}\in C^{p,1}$ as
$$\overline{\alpha_G}(g_0,\dots,g_p)(hP)=\alpha_G(h^{-1}g_0,\dots,h^{-1}g_p).$$
Now $\alpha_G$ can also be considered as an element in $L^0(G^{p+1})^{MA} \cong C^{p,2}$ which corresponds to the cocycle $\overline{\overline{\alpha_G}}\in C^{p,2}$ given by
$$\overline{\overline{\alpha_G}}(g_0,\dots,g_p)(hP,hw_0P)=\alpha_G(h^{-1}g_0,\dots,h^{-1}g_p).$$
The $M$-invariance of $\overline{\overline{\alpha_G}}$ is guaranteed by the fact that $M$ centralizes $A$ and $\pi_A$ is trivial on the maximal compact subgroup $K$. 

\begin{prop}
The isomorphisms in (\ref{iso}) are realized at a cochain level by the maps
$$\begin{array}{rcl}
C^{p+1}(A)^{w_0} &\longrightarrow & C^{p,1} \\
\alpha&\longmapsto & \overline{\alpha_G}
\end{array} 
\quad 
\mathrm{and}
\quad 
\begin{array}{rcl}
C^{p+1}(A)^{w_0} &\longrightarrow & C^{p,2} \\
\alpha&\longmapsto &\overline{\overline{\alpha_G}}
\end{array}, $$
respectively. 
\end{prop}

\begin{proof}
We start proving that the maps
$$\begin{array}{rcl}
C^{p+1}(A) &\longrightarrow & C^{p,1} \\
\alpha&\longmapsto & \overline{\alpha_G}
\end{array} 
\quad 
\mathrm{and}
\quad 
\begin{array}{rcl}
C^{p+1}(A) &\longrightarrow & C^{p,2} \\
\alpha&\longmapsto &\overline{\overline{\alpha_G}}
\end{array}, $$
induce the isomorphisms
\begin{equation}\label{iso first page}
H^p_m(A) \cong E_1^{p,1}  \quad \mathrm{and} \quad  \ H^p_m(A) \cong E_1^{p,2}.
\end{equation}
We are going to prove the statement only in the second case, that means for $C^{p,2}$, since for $C^{p,1}$ is analogous. The map $\alpha \mapsto \overline{\overline{\alpha_G}}$ is obtained as the following composition 
$$
\xymatrix{
C^{p+1}(A)  \ar[r]^{\hspace{-15pt} \iota} & L^0(A^{p+1})^A \ar[r]^{\pi_A^*} & L^0(G^{p+1})^{MA} \ar[r]^{\hspace{-20pt} \Phi} & L^0(G^{p+1},L^0(G/MA))^G.
}
$$
The function $\iota$ is the usual inclusion of cochains and it induces an isomorphism in cohomology by Austin and Moore \cite[Theorem A]{AM}. The pullback $\pi_A^*$ induces at a cohomological level the same map induced by the projection $MA \rightarrow A$, thus it is an isomorphism because $M$ is compact and centralizes $A$. The last map is an isomorphism already on cochains and it corresponds to the Eckmann-Shapiro induction, as noticed by Monod \cite[Proposition 5.1]{Monod}. Being the composition of isomorphisms in cohomology, the map $\alpha \mapsto \overline{\overline{\alpha_G}}$ induces an isomorphism as well and the claim follows. 

The explicit computation of the coboundary $d_1$  on the first page (see Section \ref{Spectral sequence}) implies that the isomorphisms \eqref{iso first page} induce the ones of Equation \eqref{iso}, and the statement is proved. 
\end{proof}

\subsection{Homotopies for $(C_K^{*,q},d^\uparrow)$ } \label{contracting homotopies} \label{homotopies}

In order to follow at cochain level the differentials realizing the injection and a section of the short exact sequence (\ref{SES}), we would ideally want contracting homotopies for the cocomplex  $(C^{*,q},d^\uparrow)$, for $q\geq 3$, whose cohomology vanishes by \cite[Proposition 5.1]{Monod}. It is however not so clear how to construct such homotopies even in the two cases we consider in this paper (for $G$ a product of isometries of hyperbolic spaces or $\mathrm{SL}(3,\mathbb{K})$). We will see that it will be enough to construct such maps on the subcocomplex of right-$K$-invariant cochains, which we now define:

 \begin{dfn} We say that a cochain $\alpha\in C^{p,q}$ is \emph{right-$K$-invariant} if
 $$\alpha(g_0,\dots,g_p)(x_1,\dots ,x_q)=\alpha(g_0k_0,\dots,g_pk_p)(x_1,\dots ,x_q)$$
 for almost every $g_0,\dots,g_p\in G$, $x_1,\dots,x_q\in G/P$ and $k_0,\dots,k_p\in K$. We define $C_K^{p,q}\subset C^{p,q}$ as the subset of right-$K$-invariant cochains. 
 \end{dfn}

\begin{rem}
It is worth noticing that the cochains $\overline{\alpha_G}$ and $\overline{\overline{\alpha_G}}$ constructed in the previous section are naturally right-$K$-invariant, since they are defined via induction of the pullback along the projection $\pi_A$ and the latter is trivial on $K$. 
\end{rem}
 
It is clear that, if $\alpha\in C_K^{p,q}$, then both $d^\rightarrow \alpha$ and $d^\uparrow \alpha$ are right-$K$-invariant. In particular $(C_K^{p,q},d^\uparrow)$ forms a cocomplex. It is easy to exhibit contracting homotopies 
$$h^{p,q}:C^{p,q}_K\longrightarrow C^{p-1,q}_K$$
for the cocomplex $(C^{p,q}_K,d^\uparrow)$, i.e. maps satisfying
\begin{equation}\label{homotopy}
h^{p+1,q}d^\uparrow + d^\uparrow h^{p,q}=\mathrm{Id}.
\end{equation} 
This could be done in general with the help of a measurable section of the projection 
$$
\xymatrix{
(G/P)^q\ar[rr] && G\backslash (G/P)^q \ar@/^1pc/[ll]_{\sigma_q}
}
$$
with the additional property that, for $q\geq 3$, the stabilizer of $\sigma(x_1,\dots,x_q)$ is contained in $K$ for almost every $q$-tuple in the image of $\sigma_q$. In this paper, we follow a different strategy, and we will exhibit contracting homotopies exploiting the transitivity of $G$ on tuples of points in $G/P$. This will be done in Section \ref{homo for products} for $G$ a product of isometries of hyperbolic spaces, which acts $3$-transitively on $G/P$, and in Section \ref{SL3R4} for $G=\mathrm{SL}(3,\mathbb{K})$, which acts more than $2$-transitively on $G/P$, meaning that it acts transitively on generic points in $G/P\times G/P\times P^1(\mathbb{K})$. Indeed recall that in this case $G/P$ is the space of flags in $\mathbb{K}^3$ and $ P^1(\mathbb{K})$ parametrizes the points in the flags.

\subsection{Explicit expression for the injection $i$ of the short exact sequence \ref{SES}}

\begin{lem} \label{Injection} Let $\{h^{*,q}\}$ be a family of contracting homotopies for $(C^{*,q}_K,d^\uparrow)$. The injection
$$i:H^{p-2}_m(A)^{w_0}\cong H^{p-2}(C^{*,2},d^\uparrow) \hookrightarrow NH^p_m(G\curvearrowright G/P)$$
induced by $d_{p-1}$ is given explicitly at cochain level as
$$\begin{array}{rcl}
C^{p-1}(A)^{w_0}&\longrightarrow & L^0((G/P)^{p+1})^G\\
  \alpha &\longmapsto & d^\rightarrow h^{1,p} d^\rightarrow h^{2,p-1} \dots d^\rightarrow h^{p-2,3} d^\rightarrow (\overline{\overline{\alpha_G}})(e).
\end{array}$$
\end{lem}


Note that the evaluation on $e$ is well-defined since an almost everywhere defined $G$-invariant function on $G$ is actually defined everywhere. Observe also that this map is defined on all cocycles, but a $w_0$-equivariant cocycle is in the image of the horizontal differential $d^\rightarrow$ and hence will be mapped to $0$ by $d^\rightarrow$ and this composition. 

\begin{proof} We only need to check that the sequence 
\begin{eqnarray*}
\alpha_2&:=&\overline{\overline{\alpha_G}} \in C^{p,2}_K,\\
\alpha_k&:=&h^{ p-k+1,k}d^\rightarrow (\alpha_{k-1}) \in C^{p-k,k}_K,
\end{eqnarray*}
 for $3 \leq k\leq p$, satisfies
\begin{equation}\label{dalpha=dalpha} d^\rightarrow \alpha_{k-1}= d^\uparrow \alpha_{k}.\end{equation} 
We will prove inductively that if $d^\uparrow d^\rightarrow \alpha_{k-1}=0$, then (\ref{dalpha=dalpha}) holds and $d^\uparrow d^\rightarrow \alpha_{k}=0$. As $\alpha_2$ is a cocycle for $d^\uparrow$, we have $d^\uparrow d^\rightarrow \alpha_{2}= d^\rightarrow d^\uparrow \alpha_{2}=0$. Suppose that $d^\rightarrow d^\uparrow\alpha_{k-1}=0$. We compute
\begin{eqnarray*}
d^\uparrow \alpha_{k}&=& d^\uparrow h^{ p-k+1,k+1}d^\rightarrow (\alpha_{k-1})\\
&=& (\mathrm{Id} - h^{p-k+2  ,k+1}d^\uparrow)d^\rightarrow (\alpha_{k-1})\\
&=& d^\rightarrow (\alpha_{k-1}),
\end{eqnarray*}
where we have used that $h^{p-k+2,k+1}d^\uparrow + d^\uparrow h^{p-k+1,k+1}=\mathrm{Id}$ and the assumption that $d^\rightarrow d^\uparrow (\alpha_{k-1})=0$. Furthermore 
$$
d^\uparrow d^\rightarrow \alpha_{k}=d^\rightarrow (d^\uparrow\alpha_{k})=d^\rightarrow (d^\rightarrow \alpha_{k-1})=0.
$$
It follows that the $d^\uparrow$-cocycle  $d^\rightarrow \alpha_p\in C^{0,p}$ represents the image of $[\alpha]$ under the differential $d_{p-1}$, which is, as an element in $L^0((G/P)^{p+1})^G$, given by evaluating the $G$-invariant cocycle  $d^\rightarrow \alpha_p$ on $e$ (or any other element of $G$). 
\end{proof}

\subsection{Explicit expression for a section $s$ of the short exact sequence (\ref{SES})}
We follow here the same strategy as for the injection in the previous paragraph. We only need an additional argument for the first step of the differential. We define 
$$ \mathcal{H}^{p-1}: C^{p-1}(A)^{w_0} \longrightarrow L^0(G^{p-1})^A\cong C^{p-2,2}$$
by    
$$\mathcal{H}^{p-1}(\overline{\alpha_G})(g_0,\dots,g_{p-2}):=$$
\vspace{-15pt}
$$(-1)^p  \sum_{i=0}^{p-2} (-1)^i \alpha_G(w_0^{-1} \pi_A(g_0),\dots ,w_0^{-1} \pi_A(g_i), w_0^{-1} g_i,\dots,w_0^{-1} g_{p-2}).$$

 \begin{lem}\label{induction} For any $\alpha\in C^{p-1}(A)^{w_0}$,  we have
 $$d^\uparrow \mathcal{H}^{p-1}(\overline{\alpha_G})=d^\rightarrow(\overline{\alpha_G}).$$
 Furthermore, $\mathcal{H}^{p-1}(\overline{\alpha_G})$ is right-$K$-invariant. 
 \end{lem}
 
 \begin{proof} The fact that $\mathcal{H}^{p-1}(\overline{\alpha_G})$ is right-$K$-invariant is obvious, since it is defined as an alternating sum of evaluations of $\alpha_G$.

It is a standard computation, based on the fact that $\alpha_G$ is a cocycle, to check that
 \begin{eqnarray*}
 d^\uparrow \mathcal{H}^{p-1}(\overline{\alpha_G})(g_0,\dots,g_{p-1})&=&(-1)^p\left(\alpha_G(w_0^{-1}g_0,\dots,w_0^{-1}g_{p-1})\right.\\
 && \left. -\alpha_G(w_0^{-1} \pi_A(g_0),\dots ,w_0^{-1} \pi_A(g_{p-1}))\right).
 \end{eqnarray*}
Now it remains to see that the latter summand is precisely $\alpha_G(g_0,\dots,g_{p-1})$, implying, by the explicit description of the differential $d^\rightarrow$ given in (\ref{d from P to A}) that $d^\uparrow \mathcal{H}^{p-1}(\overline{\alpha_G})=d^\rightarrow(\overline{\alpha_G})$. To see this, write $a_i=\pi_A(g_i)$, so that we need to check that
 $$\alpha_G(w_0^{-1}a_0,\dots,w_0^{-1} a_{p-1})=\alpha(a_0,\dots,a_{p-1}), $$
 but
 \begin{eqnarray*}
 \alpha_G(w_0^{-1}a_0,\dots,w^{-1}_0 a_{p-1})&=&\alpha(\pi_A(w_0^{-1}a_0),\dots, \pi_A(w_0^{-1}a_{p-1}))\\
 &=&\alpha(\pi_A(w_0^{-1}a_0 w_0),\dots, \pi_A(w_0^{-1}a_{p-1}w_0))\\
 &=&\alpha(w_0^{-1}a_0w_0,\dots,w_0^{-1}a_{p-1}w_0)\\
 &=&\alpha(a_0,\dots,a_{p-1}),
 \end{eqnarray*}
 by the $w_0$-invariance of $\alpha$. 
  \end{proof}

\begin{prop}\label{Section} For any $q\geq 3$, let $h^{p,q}:C^{p,q}_K\rightarrow  C^{p-1,q}_K$ be contracting homotopies for the cocomplexes $(C^{p,q}_K,d^\uparrow)$. A section of the projection in the short exact sequence from (\ref{SES}),
$$ H^{p-1}_m(A)^{w_0} \longrightarrow NH^p_m(G\curvearrowright G/P),$$
is given explicitly at cocycle level by sending a cocycle $\alpha\in C^{p}(A)^{w_0}$ to 
$$d^\rightarrow h^{1,p} d^\rightarrow h^{2,p-1} d^\rightarrow \dots d^\rightarrow h^{p-2,3} d^\rightarrow \mathcal{H}^{p-1} (\overline{\alpha_G})(e)\in L^0((G/P)^{p+1})^G.$$
\end{prop}

\begin{proof} The proof is identical to the proof of Lemma \ref{Injection} except for the first induction step which is given by Lemma \ref{induction}.
\end{proof}

\section{The spectral sequence for products of isometries of real hyperbolic spaces}\label{SS for prod of real hyp}\label{Section: Prod of hyp}

\subsection{Isometries of $\mathbb{H}^n$}\label{Isom of Hn}\label{Isom of Hn}

As before, we take the upper half space model $\{x\in \mathbb{R}^n \mid x_n>0\}$ for $\mathbb{H}^n$.  In this model the isometry group $G=\mathrm{Isom}^+(\mathbb{H}^n)$ is identified with the conformal transformations of $\mathbb{R}^n \cup \{\infty\}$  which leave the upper half space, and hence the boundary $\partial \mathbb{H}^n=\mathbb{R}^{n-1}\times \{0\} \cup \{\infty\}$, invariant. We take for maximal compact subgroup $K$ the stabilizer of $e_n=(0,\dots,0,1)$. The minimal parabolic subgroup $P=\mathrm{Stab}(\infty)$ is given by the similarities of $\mathbb{R}^{n-1}$ (which naturally extend to te upper half space). More precisely, $P=ANM$, where $A$ consists of the homoteties $a_\lambda:x\mapsto \lambda x$, for $\lambda>0$, $N$ consists of translations $n_v:x\mapsto x+v$ by vectors $v$ in $\mathbb{R}^{n-1}\times \{0\}$ and $M=P\cap K=\mathrm{SO}(n-1)$  are orthogonal transformations (extending to orthogonal transformations of $\mathbb{R}^n$ which restrict to $\mathbb{H}^n$). We will need an explicit description of a representative of the longest element $w_0\in K$, which  can be taken as the composition of the inversion with the unit sphere centered at the origin and the reflexion by the hyperplane orthogonal to $e_1=(1,0,\dots,0)$ and thus takes the form
\begin{equation}\label{w0}\begin{array}{rccl}
w_0:&\mathbb{R}^n&\longrightarrow & \mathbb{R}^n\\
   &x&\longmapsto & \frac{x-2\langle x,e_1\rangle e_1}{\|x\|^2}.
   \end{array}\end{equation}
   
We record the following commutativity rules:
\begin{equation}\label{eq:commutativity:law}
a_\lambda n_v a_{\frac{1}{\lambda}}=n_{\lambda v} , \ \ a_{\lambda} \rho = \rho a_{\lambda} , \ \ \rho n_v \rho^{-1}=n_{\rho v} ,
\end{equation}
for $a_\lambda \in A, n_v \in N$ and $\rho \in M$. Additionally, setting $n:=n_{e_1}$, we observe the relation
\begin{equation}\label{6 terms rel}
w_0^{-1}nw_0^{-1}=n^{-1}w_0^{-1}n^{-1}.\end{equation}

By the Iwasawa decomposition, we have a diffeomorphism between $G$ and $ANK$, which allows us to define the projection 
$$\begin{array}{rccl}
\pi_A:&G &\longrightarrow &A\\
&g=ank&\longmapsto &a. 
\end{array}$$

The following simple lemma will allow us  to compute the $\pi_A$-projection of any element of $G$: 

\begin{lem}\label{lem:projection:last:coordinate} For any $g\in G$, its $\pi_A$-projection is
$$\pi_A(g)=a_{\lambda},
$$
where 
$$
 (0,\infty)\ni \lambda=\textup{the last coordinate of the vector $g(e_n)$} .
$$
\end{lem}

\begin{proof}
Since $K$ stabilizes $e_n$, only $N$ and $A$ act non trivially on that point. In particular, if $g=a_\lambda n_v \rho$, for $\rho\in K$, $v\in \mathbb{R}^{n-1}\times \{0\}$ and $\lambda>0$, then 
$$
g(e_n)=a_\lambda n_v \rho(e_n)=a_\lambda n_v (e_n)=a_\lambda(v+e_n)=\lambda v+\lambda e_n, 
$$
which proves the lemma since the last coordinate of $v$ is $0$. \end{proof}

We apply this lemma to compute a collection of $\pi_A$-projections needed in the proofs of Theorems \ref{Theorem 2 to 3 for prod of Hn} and \ref{Theorem 2 to 4 for prod of Hn} to some specific elements in $G$ which we now define. For $x\in \partial \mathbb{H}^n\setminus \{\infty,0\}$, denote by $\widehat{x}\in \partial \mathbb{H}^n$ the normalized element
$$
\widehat{x}=\frac{x}{\|x\|}.
$$
In case $n=2$, we further assume that $0<x<\infty$. Now choose a (non unique in dimension $n\geq 3$) isometry $\rho_{\widehat{x}} \in M$ sending $e_1$ to $\widehat{x}$. Set 
\begin{align}\label{eq:three:group:elements}
g_{\infty0x}&:=a_{\|x\|}\rho_{\widehat{x}}  , \\ 
g_{\infty e_1x}&:=na_{\|x-e_1\|}\rho_{\widehat{x-e_1}} , \nonumber \\
g_{0 e_1x}&:=w_0n^{-1}a_{\|e_1+w_0(x)\|}\rho_{\widehat{e_1+w_0(x)}} \nonumber  .
\end{align}
These isometries are so defined as to send the three points $\infty,0,e_1$ to the three points listed as indices. For example, $g_{0 e_1x} (\infty)=0    , g_{0 e_1x} (0)=e_1    , g_{0 e_1x}(e_1)=x$.

The next proposition will be needed in the proof of Theorem \ref{Theorem 2 to 3 for prod of Hn}.

\begin{prop}\label{prop: pi A projections} For $x\in \partial \mathbb{H}^n\setminus \{\infty,0\}$ and $g_{\infty0x},g_{\infty e_1x},g_{0 e_1x}$ as above we have:
$$
\pi_A(g_{0 e_1x}^{-1})=a_{\frac{\|x\|}{\|x-e_1\|}}, \ \ \pi_A(w_0^{-1}g_{0e_1x}^{-1})=a_{\frac{\|x-e_1\|}{2\|x\|}} , \ \ \pi_A(w_0^{-1}n^{-1}g_{0e_1x}^{-1})  =a_{\frac{\|x\|\|x-e_1\|}{1+\|x\|^2}},
$$
$$
\pi_A(g_{\infty e_1  x}^{-1})=a_{\frac{1}{\|x-e_1\|}}, \ \ \pi_A(w_0^{-1}g_{\infty e_1x}^{-1})=a_{\frac{\|x-e_1\|}{2}} , \ \ \pi_A(w_0^{-1}n^{-1}g_{\infty e_1x}^{-1})=a_{\frac{\|x-e_1\|}{1+\|x\|^2}} ,
$$
$$
\pi_A(g_{\infty0x}^{-1})=a_{\frac{1}{\|x\|}} , \ \ \pi_A(w_0^{-1}g_{\infty0x}^{-1})=a_{\|x\|}  , \ \ \pi_A(w_0^{-1}n^{-1}g_{\infty0x}^{-1})=a_{\frac{\|x\|}{1+\|x\|^2}}  .
$$
\end{prop}
Note in particular that these projections do not depend on the choice of $\rho_{\widehat{x}} \in M$.

\begin{proof} We only prove the first line of equations, i.e those involving $g_{0 e_1x}$, which are the most tricky. The motivated reader can check the remaining easier equations using the same method. From the definition of $g_{0 e_1x}$ we obtain 
\begin{align*}
g^{ - 1}_{0 e_1x}=&\rho^{-1}_{\widehat{e_1+w_0(x)}}a_{\frac{1}{\|e_1+w_0(x)\|}}nw_0^{-1}\\
=&a_{\frac{1}{\|e_1+w_0(x)\|}}\rho^{-1}_{\widehat{e_1+w_0(x)}}nw_0^{-1}\\
=&a_{\frac{1}{\|e_1+w_0(x)\|}}n_{\rho^{-1}_{\widehat{e_1+w_0(x)}}(e_1)}\rho^{-1}_{\widehat{e_1+w_0(x)}}w_0^{-1},
\end{align*}
where we used twice the commutativity rules given in (\ref{eq:commutativity:law}). Since this expression is of the form $ank$ we have that 
$$
\pi_A(g_{0 e_1x}^{-1})=a_{\frac{1}{\|e_1+w_0(x)\|}}  .
$$
The square of the norm of the element $e_1+w_0(x)$ is easily computed:
\begin{align*}
\|e_1+w_0(x)\|^2=&\|e_1\|^2+\|w_0(x)\|^2+2\langle e_1,w_0(x) \rangle=\\
=&1+\frac{1}{\|x\|^2}-2\frac{\langle e_1,x \rangle}{\|x\|^2}=\\
=&\frac{\|x-e_1\|^2}{\|x\|^2} .
\end{align*}
As a consequence we get the first expression
$$
\pi_A(g_{0e_1x}^{-1})=a_{\frac{\|x\|}{\|x-e_1\|}}
$$
as expected. Starting from the above expression for $g_{0e_1x}^{-1}$ and using the fact that $w_0^{-1}aw_0=a^{-1}$ for every $a\in A$, we get that
$$
\pi_A(w_0^{-1}g_{0e_1x}^{-1})=a_{\frac{\|x-e_1\|}{\|x\|}}\pi_A(w_0^{-1}n_{\rho^{-1}_{\widehat{e_1+w_0(x)}}(e_1)}).
$$
By Lemma \ref{lem:projection:last:coordinate} we know that
$$
\pi_A(w_0^{-1}n_v)=a_{\frac{1}{1+\|v\|^2}}  ,
$$
for any $v\in \mathbb{R}^{n-1}\times \{0\}$. In this particular context the previous formula allows us to write
$$
\pi_A(w_0^{-1}g_{0e_1x}^{-1})=a_{\frac{\|x-e_1\|}{\|x\|}}a_{\frac{1}{2}}=a_{\frac{\|x-e_1\|}{2\|x\|}} ,
$$
which is the second desired expression. We now compute the third term
\begin{align*}
 \pi_A(w_0^{-1}n^{-1}g_{0e_1x}^{-1})=&\pi_A(w_0^{-1}n^{-1}a_{\frac{1}{\|e_1+w_0(x)\|}}n_{\rho^{-1}_{\widehat{e_1+w_0(x)}}(e_1)})\\
 =&a_{\|e_1+w_0(x)\|}\pi_A(w_0^{-1}n_{-\|e_1+w_0(x)\|e_1+\rho^{-1}_{\widehat{e_1+w_0(x)}}(e_1)}).
\end{align*}
where we used again the relations from \eqref{eq:commutativity:law} to pull out the dilation term. To conclude the computation, we need to consider the square of the norm of the vector appearing in the translation term. We have that
\begin{align*}
\left\|\|e_1+w_0(x)\|e_1-\rho^{-1}_{\widehat{e_1+w_0(x)}}(e_1)\right\|^2=&\|e_1+w_0(x)\|^2+1-2\|e_1+w_0(x)\|\frac{\langle e_1+w_0(x),e_1 \rangle}{\|e_1+w_0(x)\|}\\
=&1+\frac{1}{\|x\|^2}+2\langle e_1,w_0 \rangle +1-2-2\langle e_1,w_0(x) \rangle\\
=&\frac{1}{\|x\|^2}  .
\end{align*}
Therefore, the last projection is given by
$$
 \pi_A(w_0^{-1}n^{-1}g_{0e_1x}^{-1})=a_{\frac{\|x-e_1\|}{\|x\|}}a_{\frac{1}{1+\frac{1}{\|x\|^2}}}=a_{\frac{\|x\|\|x-e_1\|}{1+\|x\|^2}} .
$$

\end{proof}

For the proof of Theorem \ref{Theorem 2 to 4 for prod of Hn} we will need the following three propositions:

\begin{prop} \label{piA for mu}Take $x\in \partial \mathbb{H}^n\setminus \{\infty,0\}$ and $g_{\infty0x},g_{\infty e_1x},g_{0 e_1x}$ as above. For any $y\in \partial \mathbb{H}^n\setminus \{\infty,0,x\}$ set 
$$\mu(x,y):= g^{-1}_{\infty e_1 x}(y).$$
We have
$$\begin{array}{cc}
\pi_A(w_0^{-1}g_{0e_1\mu(x,y)}^{-1})=a_{\frac{\|y-x\|}{2\|y-e_1\|}},&\pi_A(w_0^{-1}g_{{0}e_1\mu(x,y)}^{-1}g_{{\infty}  {e_1} x}^{-1})=a_{\frac{\|y-x\|\|x-e_1\|}{\|y-e_1\|(\|x\|^2+1)}}, \\
\pi_A(w_0^{-1}g_{\infty e_1 \mu(x,y)}^{-1})=a_{\frac{\|y-x\|}{2\|x-e_1\|}},& \pi_A(w_0^{-1}g_{{\infty} e_1\mu(x,y)}^{-1}g_{{\infty}  {e_1} x}^{-1})=a_{\frac{\|y-x\|}{\|x\|^2+1}} ,\\
\pi_A(w_0^{-1}g_{\infty 0 \mu(x,y)}^{-1})=a_{\frac{\|y-e_1\|}{\|x-e_1\|}},&\pi_A(w_0^{-1}g_{{\infty}{0}\mu(x,y)}^{-1}g_{{\infty} {e_1} x}^{-1})=a_{\frac{\|y-e_1\|}{2}}.
\end{array}$$
\end{prop}

\begin{proof} For the three left projections, we will use the middle equations of Proposition \ref{prop: pi A projections} applied to $\mu(x,y)$. We thus need to understand how  we can write the point $\mu(x,y)$. Recall that $\widehat{y}=y/\|y\|$ is the normalization of $y$ and the maps $g_{\infty 0x},g_{\infty e_1x},g_{0e_1x}$ are as in \eqref{eq:three:group:elements}. We have
\begin{align*}
\mu(x,y)=g_{\infty e_1 x}^{-1}(y)=&a_{\frac{1}{\|x-e_1\|}}\rho^{-1}_{\widehat{x-e_1}}n^{-1}y=\\
=&a_{\frac{\|y-e_1\|}{\|x-e_1\|}}\rho^{-1}_{\widehat{x-e_1}}(\widehat{y-e_1}).
\end{align*}
Let us further compute the square of the norm of $\mu(x,y)-e_1$:
\begin{align*}
\|\mu(x,y)-e_1\|^2=&\left\|\frac{\|y-e_1\|}{\|x-e_1\|}\rho^{-1}_{\widehat{x-e_1}}(\widehat{y-e_1})-e_1\right\|^2=\\
=&\frac{\|y-e_1\|^2}{\|x-e_1\|^2}+1-2\frac{\|y-e_1\|}{\|x-e_1\|}\frac{\langle y-e_1,x-e_1 \rangle}{\|x-e_1\|\|y-e_1\|}=\\
=&\frac{\|y-x\|^2}{\|x-e_1\|^2} . 
\end{align*}
Proposition \ref{prop: pi A projections} now gives: 
\begin{align*}
\pi_A(w_0^{-1}g_{0e_1\mu(x,y)}^{-1})=&a_{\frac{\|\mu(x,y)-e_1\|}{\|\mu(x,y)\|}}a_{\frac{1}{2}}=a_{\frac{\|y-x\|}{2\|y-e_1\|}} \\
\pi_A(w_0^{-1}g_{\infty e_1 \mu(x,y)}^{-1})=&a_{\frac{\|y-x\|}{2\|x-e_1\|}} \\ 
\pi_A(w_0^{-1}g_{\infty 0 \mu(x,y)}^{-1})=&a_{\frac{\|y-e_1\|}{\|x-e_1\|}}  .
\end{align*}
For the projection of the right element of the first line, we have that 
\begin{align*}
w_0^{-1}g_{0 e_1 \mu(x,y)}^{-1}g_{\infty e_1 x}^{-1}=&w_0^{-1}a_{\frac{1}{\|e_1+w_0(\mu(x,y))\|}}\rho^{-1}_{\widehat{e_1+w_0(\mu(x,y))}}nw_0^{-1}a_{\frac{1}{\|x-e_1\|}}\rho^{-1}_{\widehat{x-e_1}}n^{-1}=\\
=&a_{\frac{\|e_1+w_0(\mu(x,y))\|}{\|x-e_1\|}}w_0^{-1}\rho^{-1}_{\widehat{e_1+w_0(\mu(x,y))}}n_{\frac{e_1}{\|x-e_1\|}}w_0^{-1}\rho^{-1}_{\widehat{x-e_1}}n^{-1} , 
\end{align*}
where we used Equation \eqref{eq:commutativity:law} to move from the first line to the second one. It follows that
$$\pi_A(w_0^{-1}g_{0 e_1 \mu(x,y)}^{-1}g_{\infty e_1 x}^{-1})= a_{\frac{\|e_1+w_0(\mu(x,y))\|}{\|x-e_1\|}}\pi_A(w_0^{-1}\rho^{-1}_{\widehat{e_1+w_0(\mu(x,y))}}n_{\frac{e_1}{\|x-e_1\|}}w_0^{-1}\rho^{-1}_{\widehat{x-e_1}}n^{-1}).$$
To compute the projection of the latter element, we need to consider its evaluation on $e_n$ and compute its last coordinate according to Lemma \ref{lem:projection:last:coordinate}:
\begin{align*}
&\left(w_0^{-1}\rho^{-1}_{\widehat{e_1+w_0(\mu(x,y))}}n_{\frac{e_1}{\|x-e_1\|}}w_0^{-1}\rho^{-1}_{\widehat{x-e_1}}n^{-1}\right)(e_n)=\\
=&\left[w_0^{-1}\rho^{-1}_{\widehat{e_1+w_0(\mu(x,y))}}\left(\frac{1}{2}e_n-\frac{1}{2}\rho^{-1}_{\widehat{x-e_1}}(e_1)+\frac{\langle x,e_1 \rangle}{\|x-e_1\|}\right)\right]  .
\end{align*}
Since $K$ stabilizes the point $e_n$, the last coordinate will be given by $1/2$ divided by the square of the norm of the vector inside the round brackets. If we compute the square of this norm we obtain
\begin{align*}
\left\|\frac{1}{2}e_n-\frac{1}{2}\rho_{x-e_1}^{-1}(e_1)+\frac{\langle x,e_1 \rangle}{\|x-e_1\|} \right\|^2=&\frac{1}{2}+\frac{\langle x,e_1 \rangle^2}{\|x-e_1\|^2}-\frac{\langle x,e_1 \rangle}{\|x-e_1\|}\frac{\langle e_1,x-e_1 \rangle}{\|x-e_1\|}\\
=&\frac{1}{2}+\frac{\langle x,e_1 \rangle}{\|x-e_1\|^2}=\frac{\|x\|^2+1}{2\|x-e_1\|^2}.
\end{align*}
Finally we can conclude that 
$$\pi_A(w_0^{-1}g_{0 e_1 \mu(x,y)}^{-1}g_{\infty e_1 x}^{-1})= a_{\frac{\|e_1+w_0(\mu(x,y))\|}{\|x-e_1\|}}a_{\frac{\|x-e_1\|}{\|x\|^2+1}}=a_{\frac{\|y-x\|\|x-e_1\|}{\|y-e_1\|(\|x\|^2+1)}} .$$

Since the computations of the two remaining projections are entirely similar, we omit the details.\end{proof}

We will need to compute 12 further projections which we simply report now without proof. The details of the computations follow the same line as the proof of Proposition \ref{piA for mu}.

\begin{prop}\label{piA for lambda} Take $x\in \partial \mathbb{H}^n\setminus \{\infty,0\}$ and $g_{\infty0x},g_{\infty e_1x},g_{0 e_1x}$ as above. For any $y\in \partial \mathbb{H}^n\setminus \{\infty,0,x\}$ set 
$$\lambda(x,y):= g^{-1}_{\infty 0 x}(y).$$
We have
$$\begin{array}{cc}
\pi_A(w_0^{-1}g_{0e_1\lambda(x,y)}^{-1})=a_{\frac{\|y-x\|}{2\|y\|}},& \pi_A(w_0^{-1}g_{\infty e_1\lambda(x,y)}^{-1})=a_{\frac{\|y-x\|}{2\|x\|}}, \\
\pi_A(w_0^{-1}g_{\infty 0 \lambda(x,y)}^{-1})=a_{\frac{\|y\|}{\|x\|}},& \pi_A(w_0^{-1}g_{0e_1\lambda(x,y)}^{-1}g_{\infty 0 x}^{-1})=a_{\frac{\|x\|\|y-x\|}{\|y\|(1+\|x\|^2)}} , \\
\pi_A(w_0^{-1}g_{\infty e_1 \lambda(x,y)}^{-1}g_{\infty 0 x}^{-1})=a_{\frac{\|y-x\|}{\|x\|^2+1}},
&\pi_A(w_0^{-1}g_{\infty 0 \lambda(x,y)}^{-1}g_{\infty0x}^{-1})=a_{\|y\|}.
\end{array}
$$
\end{prop}
\begin{prop}\label{piA for nu} Take $x\in \partial \mathbb{H}^n\setminus \{\infty,0\}$ and $g_{\infty0x},g_{\infty e_1x},g_{0 e_1x}$ as above. For any $y\in \partial \mathbb{H}^n\setminus \{\infty,0,x\}$ set 
$$\nu(x,y):= g^{-1}_{0 e_1 x}(y).$$
We have
$$\begin{array}{cc}
$$
\pi_A(w_0^{-1}g_{0e_1\nu(x,y)}^{-1})=a_{\frac{\|y-x\|}{2\|x\|\|y-e_1\|}},&\pi_A(w_0^{-1}g_{\infty e_1 \nu(x,y)}^{-1})=a_{\frac{\|y-x\|}{2\|y\|\|x-e_1\|}}, \\
\pi_A(w_0^{-1}g_{\infty0\nu(x,y)}^{-1})=a_{\frac{\|x\|\|y-e_1\|}{\|y\|\|x-e_1\|}} ,
&
\pi_A(w_0^{-1}g_{0e_1\nu(x,y)}^{-1}g_{0e_1 x}^{-1})=a_{\frac{\|y-x\|\|x-e_1\|}{\|y-e_1\|(1+\|x\|^2)}}, \\  \pi_A(w_0^{-1}g_{\infty e_1\nu(x,y)}^{-1}g_{0e_1 x}^{-1})=a_{\frac{\|y-x\|\|x\|}{\|y\|(1+\|x\|^2)}} ,
&
\pi_A(w_0^{-1}g_{\infty 0 \nu(x,y)}^{-1}g_{0e_1 x}^{-1})=a_{\frac{\|y-e_1\|}{2\|y\|}}.\\
\end{array}$$
\end{prop}

\subsection{Contracting homotopies and differentials}\label{homo for products} Let now $G=G_1\times \dots \times G_k$, where $G_i=\mathrm{Isom}^+(\mathbb{H}^{n_i})$ for $n_i\geq 2$. A representative for the longest element $w_0\in G$ can be taken to be on each factor the element $w_0$ defined in (\ref{w0}) in the case of one factor. We hope that the confusing notation
$$w_0=(w_0,\dots,w_0)\in K$$
will not confuse the reader. Likewise, we define $n\in G$ as the translation by $e_1$ on each factor, so that here also we write
$$n=(n,\dots,n)\in P.$$
We are more careful with elements in $G/P$ and write
\begin{align*}
\underline{    \infty }=&(\infty,\dots,\infty),\\
\underline{    0 }=&(0,\dots,0),\\
\underline{   1  }=&(e_1,\dots,e_1),
\end{align*}
for our favorite three elements in the boundary $G_1/P_1\times \dots \times G_k/P_k=G/P$.

Contracting homotopies for the cocomplexes $(C^{*,q}_K,d^\uparrow)$, for $q\geq 3$, can obviously be defined as follows: 
$$\begin{array}{rccl}
h^{p,q}:&C^{p,q}_K&\longrightarrow & C^{p-1,q}_K\\
\end{array}
$$ 
with 
$$h^{p,q}(f)(g_0,\dots, g_{p-1})(\underline{\infty}, \underline{0}, \underline{1}, x_4,\dots, x_q):= f(e,g_0,\dots, g_{p-1})(\underline{\infty}, \underline{0}, \underline{1}, x_4,\dots, x_q).$$
Note that this defines a measurable function on a dense subset of $G^p$: Indeed, although $f$ is only defined on a subset of full measure of $G^{p+1}$, the evaluation on $(e,g_0,\dots,g_{p-1})$ makes sense on a subset of full measure of $G^p$ by the $G$-invariance of $f$. Likewise, the evaluation on $(\underline{\infty}, \underline{0}, \underline{1}, x_4,\dots, x_q)$ is allowed for almost all $(x_4,\dots,x_q)$ by $3$-transitivity. Finally observe that the resulting cochain is $G$-invariant by construction and $3$-transitivity. 

It is straightforward to check that 
$$h^{p+1,q}\circ d^\uparrow + d^\uparrow \circ h^{p,q}=\mathrm{Id}.$$

Before we proceed, we can again exploit $3$-transitivity to give the following explicit form for the differential $d^\rightarrow:C^{p,2}\rightarrow C^{p,3}$:

\begin{lem}\label{d 2 to 3} The differential
$$d^\rightarrow:C^{p,2}=L^0(G^{p+1})^A \longrightarrow L^0(G^{p+1})=C^{p,3}$$
is given by
$$d^\rightarrow(\beta)(g_0,\dots,g_p)=(-1)^{p+1}[\beta(nw_0^{-1}g_0,\dots,nw_0^{-1}g_p)-\beta(n^{-1}g_0,\dots,n^{-1}g_p)+\beta(g_0,\dots,g_p)].$$
\end{lem}

\begin{proof} A cochain $\beta\in L^0(G^{p+1})^A$ corresponds to $  \overline{\beta}\in C^{p,2}=L^0(G^{p+1}, L((G/P)^2))^G$ in the following way:
$$\overline{\beta}(g_0,\dots,g_p)(h^{-1}\underline{\infty},h^{-1} \underline{0}):=\beta(hg_0,\dots,hg_p),$$
Likewise the cochain $d^\rightarrow(\beta)\in L^0(G^{p+1})$ is given by
$$d^\rightarrow(\beta)(g_0,\dots,g_p)=d^\rightarrow \overline{\beta}(g_0,\dots,g_p)(\underline{\infty}, \underline{0},\underline{1}),$$
where the evaluation on the triple $(\underline{\infty}, \underline{0},\underline{1})$ makes sense by the $G$-invariance of $d^\rightarrow \overline{\beta}$. By definition of $d^\rightarrow$, we have
\begin{align*}d^\rightarrow \overline{\beta}(g_0,\dots,g_p)(\underline{\infty}, \underline{0},\underline{1})=& (-1)^{p+1}[\overline{\beta}(g_0,\dots,g_p)(\underline{0},\underline{1})-\overline{\beta}(g_0,\dots,g_p)(\underline{\infty},\underline{1})+\overline{\beta}(g_0,\dots,g_p)(\underline{\infty}, \underline{0})]\\
=&(-1)^{p+1}[\overline{\beta}(g_0,\dots,g_p)(w_0n^{-1}\underline{\infty}, w_0 n^{-1}\underline{0})-\overline{\beta}(g_0,\dots,g_p)(n\underline{\infty}, n\underline{0})\\
&+\overline{\beta}(g_0,\dots,g_p)(\underline{\infty}, \underline{0})]\\
=&(-1)^{p+1}[{\beta}(nw_0^{-1}g_0,\dots,nw_0^{-1}g_p)-{\beta}(n^{-1}g_0,\dots,n^{-1}g_p)\\
&+{\beta}(g_0,\dots,g_p)],
\end{align*}
which proves the lemma. 
\end{proof}

\subsection{Proof of Theorem \ref{Theorem 2 to 3 for prod of Hn}}\label{proof of theorem 2}

Let $\alpha_\mathfrak{a} \in \wedge^2\mathfrak{a}^\ast$ be an alternating form, $\widetilde{\alpha}:A^2\rightarrow \mathbb{R}$ be the corresponding inhomogenous $2$-cocycle 
$$\widetilde{\alpha}(a_1,a_2):= \alpha_\mathfrak{a}(\log a_1,\log a_2)$$
and $\alpha:A^3\rightarrow \mathbb{R}$ be the homogenous $A$-invariant cocycle 
$$\alpha(a_0,a_1,a_2)=\widetilde{\alpha}(a_0^{-1}a_1,a_1^{-1}a_2).$$

The cocycle
$$\alpha_G:G^3\longrightarrow \mathbb{R}$$
defined by
$$\alpha_G(g_0,g_1,g_2)=\alpha(\pi_A(g_0),\pi_A(g_1),\pi_A(g_2)),$$
where $\pi_A:G\rightarrow A$ is as above the projection given by the Iwasawa decomposition, 
is clearly $P$-invariant, so is an element in $L^0(G^3)^P=C^{2,1}$. Furthermore, its cohomology class in $H^2(C^{2,1},d^\uparrow)\cong H^2_m(P)\cong H^2_m(A)\cong (\wedge^2 \mathfrak{a})^*$ clearly corresponds to $\alpha_\mathfrak{a}$. By Proposition \ref{Section} we need to follow $\overline{\alpha_G}$ through the maps
\begin{equation}\label{eq:diagram:chasing}
\xymatrix{
\overline{\alpha_G} \in C^{2,1} \ar[r] & d^\rightarrow \overline{\alpha_G}=d^\uparrow \beta \in C^{2,2} & &\\
& \mathcal{H}^2(\overline{\alpha_G})=:\beta \in C^{1,2} \ar[u] \ar[r] & d^\rightarrow \beta=d^\uparrow \omega\in C^{1,3} &\\
& & h^{1,3}(d^\rightarrow \beta)=: \omega \in C^{0,3} \ar[u] \ar[r] &d^\rightarrow \omega=:\Omega_3 \in C^{0,4} 
}
\end{equation}
and evaluate $\Omega_3=d^\rightarrow \omega$ on $e$. 

\subsubsection*{Computation of $\beta$} By the definition of $\mathcal{H}^2$, we have 
\begin{eqnarray*}\label{beta}\beta(g_0,g_1)&=&\mathcal{H}^2(\overline{\alpha_G})(g_0,g_1)\\
&=&-\alpha_G(w_0^{-1}\pi_A(g_0),w_0^{-1}g_0,w_0^{-1}g_1)+ \alpha_G(w_0^{-1}\pi_A(g_0),w_0^{-1}\pi_A(g_1),w_0^{-1}g_1).
\end{eqnarray*}

\subsubsection*{Computation of $\omega$} By the description of $d^\rightarrow$ from Lemma \ref{d 2 to 3} and the definition of $h^{1,3}$ we obtain $\omega\in C^{0,3}=L^0(G,\mathbb{R})$ as 
\begin{eqnarray*}
\omega(g)&=&h^{1,3}(d^\rightarrow \beta )=d^\rightarrow \beta(e,g)\\
	&=&\beta(nw_0^{-1},nw_0^{-1}g)-\beta(n^{-1},n^{-1}g_0)+\beta(e,g).
	\end{eqnarray*}
Using the expression for $\beta$ above, we can rewrite it as
\begin{eqnarray}\label{eq:computation:omega}
\omega(g)&=&-\alpha_G(w_0^{-1}\pi_A(nw_0^{-1}),w_0^{-1}nw_0^{-1},w_0^{-1}nw_0^{-1}g)\\
&&+ \alpha_G(w_0^{-1}\pi_A(nw_0^{-1}),w_0^{-1}\pi_A(nw_0^{-1}g),w_0^{-1}n w_0^{-1}g)\nonumber\\
&&+\alpha_G(w_0^{-1}\pi_A(n^{-1}),w_0^{-1}n^{-1},w_0^{-1}n^{-1}g)\nonumber\\
&&-\alpha_G(w_0^{-1}\pi_A(n^{-1}),w_0^{-1}\pi_A(n^{-1}g),w_0^{-1}n^{-1}g )\nonumber\\
&&-\alpha_G(w_0^{-1}\pi_A(e),  w_0^{-1}    ,w_0^{-1}g) \nonumber  \\
&&+ \alpha_G(w_0^{-1}\pi_A(e),  w_0^{-1}\pi_A(g)   ,w_0^{-1}g). \nonumber
	\end{eqnarray}
Since $\pi_A(nw_0^{-1})=\pi_A(n^{-1})=\pi_A(w_0^{-1})=e$ the first coordinate of these $6$ evaluations of $\alpha_G$ is the identity $e$. Additionally, since $w_0^{-1}nw_0^{-1}=n^{-1}w_0^{-1}n^{-1}$, we have, for $g'=e $ or $g$ (or any element in $G$) that
\begin{equation}\label{pi of nwn}\pi_A(w_0^{-1}nw_0^{-1} g')=\pi_A(n^{-1}w_0^{-1}n^{-1}g')=\pi_A(w_0^{-1}n^{-1}g'),\end{equation}
where for the last equality we have used that $\pi_A(n'g')=\pi_A(g')$ for any $n'\in N$ and $g'\in G$. It follows that the first and third summands in Equation \eqref{eq:computation:omega} have precisely the same coordinates and hence  cancel. Moreover, the fifth summand vanishes since the evaluation of $\alpha_G$ is zero whenever two of the coordinates are equal (and here the first and second coordinates are $e$). We are thus left with the $2$nd, $4$th and $6$th summands, which, using (\ref{pi of nwn}) and the $N$-(left) invariance of $\pi_A$, we rewrite as 
\begin{eqnarray*}
\omega(g)&=&\alpha_G(e,w_0^{-1}\pi_A(w_0^{-1}g),w_0^{-1}n^{-1}g)\\
&&-\alpha_G(e,w_0^{-1}\pi_A(g),w_0^{-1}n^{-1}g)\\
&&+\alpha_G(e,w_0^{-1}\pi_A(g),w_0^{-1}g).
\end{eqnarray*}
Now recall on the one hand that $\pi_A(w_0^{-1}a)=a^{-1}$ and on the other hand $\alpha(e,a_1,a_2)=-\alpha(e,a_1,(a_2)^{-1})$ and hence 
$$\alpha_G(e,g_1,g_2)=-\alpha_G(e,g_1,\pi_A(g_2)^{-1}),$$
to conclude that
\begin{eqnarray*}
\omega(g)&=&-\alpha_G(e,\pi_A(w_0^{-1}g)^{-1},\pi_A(w_0^{-1}n^{-1}g)^{-1})\\
&&+\alpha_G(e,w_0^{-1}\pi_A(g),\pi_A(w_0^{-1}n^{-1}g)^{-1})\\
&&-\alpha_G(e,w_0^{-1}\pi_A(g),\pi_A(w_0^{-1}g)^{-1})\\
&=&d^\uparrow \alpha_G(e,w_0^{-1}\pi_A(g),\pi_A(w_0^{-1}g)^{-1},\pi_A(w_0^{-1}n^{-1}g)^{-1})\\
&&-\alpha_G(w_0^{-1}\pi_A(g),\pi_A(w_0^{-1}g)^{-1},\pi_A(w_0^{-1}n^{-1}g)^{-1})\\
&=&-\alpha_G(g,w_0^{-1}g,w_0^{-1}n^{-1}g) \,
\end{eqnarray*}
which corresponds to the value of $\omega(g)$ on our favorite triple of points. Notice that we exploited the $w_0$-invariance of $\alpha_G$ to obtain the last equation. In this way we finally get
\begin{equation}\label{omega}\omega(g) (\underline{\infty},\underline{0},\underline{e_1})=-\alpha_G(g,w_0^{-1}g,w_0^{-1}n^{-1}g) .\end{equation}

\subsubsection*{Computation of $\Omega_3=d^\rightarrow \omega(e)$} This final step will finally give us the representative $\Omega_3$ of the image of $\alpha_\mathfrak{a}$ under the section $s$. By definition of $d^\rightarrow$ we have
\begin{align*}
\Omega_3({\underline{\infty}},\underline{0},\underline{e_1},\underline{x})&=d^\rightarrow \omega(e)({\underline{\infty}},\underline{0},\underline{e_1},\underline{x})\\
&=-\omega(e)(\underline{0},\underline{e_1},\underline{x})+\omega(e)(\underline{\infty},\underline{e_1},\underline{x})-\omega(e)(\underline{\infty},\underline{0},\underline{x})+\omega(e)(\underline{\infty},\underline{0},\underline{e_1})\ .
\end{align*}
The negative sign is again due to the weight used to define $d^\rightarrow$. By our computation of $\omega$ in (\ref{omega}), we immediately see that the fourth summand vanishes. In order to compute the remaining three summands, since we only know the value of $\omega(g)$ when evaluated on $(\underline{\infty},\underline{0},\underline{e_1})$, we need to use transitivity and the invariance of $\omega$ to replace each summand by an appropriate evaluation on this particular triple. More precisely, given any triple $(\underline{x},\underline{y},\underline{z})\in (G/P)^3$,  we can choose $g_{\underline{x}\underline{y}\underline{z}}\in G$ to be an element such that 
$$g_{\underline{x}\underline{y}\underline{z}}.\underline{\infty}=\underline{x}, \ g_{\underline{x}\underline{y}\underline{z}}.\underline{0}=\underline{y} \ \mathrm{and} \ g_{\underline{x}\underline{y}\underline{z}}.\underline{e_1}=\underline{z}.$$ 
In particular, for any such choice we obtain
\begin{eqnarray*}
\Omega_3({\underline{\infty}},\underline{0},\underline{e_1},\underline{x})&=&-\omega(g_{\underline{0}\underline{e_1}\underline{x}}^{-1})+\omega(g_{\underline{\infty}\underline{0}\underline{x}}^{-1})-\omega(g_{\underline{\infty}\underline{0}\underline{x}}^{-1})= \\
&=&\alpha_G(g_{\underline{0}\underline{e_1}\underline{x}}^{-1},w_0^{-1}g_{\underline{0}\underline{e_1}\underline{x}}^{-1},w_0^{-1}n^{-1}g_{\underline{0}\underline{e_1}\underline{x}}^{-1}) \\
&&-\alpha_G(g_{\underline{\infty}\underline{0}\underline{x}}^{-1},w_0^{-1}g_{\underline{\infty}\underline{0}\underline{x}}^{-1},w_0^{-1}n^{-1}g_{\underline{\infty}\underline{0}\underline{x}}^{-1}) \\
&&+\alpha_G(g_{\underline{\infty}\underline{0}\underline{x}}^{-1},w_0^{-1}g_{\underline{\infty}\underline{0}\underline{x}}^{-1},w_0^{-1}n^{-1}g_{\underline{\infty}\underline{0}\underline{x}}^{-1}) .
\end{eqnarray*}

We have already considered choices of such isometries $g_{\underline{0}\underline{e_1}\underline{x}}, g_{\underline{\infty}\underline{0}\underline{x}} ,g_{\underline{\infty}\underline{0}\underline{x}}$ in Section \ref{Isom of Hn} (\ref{eq:three:group:elements}) on each factor and further computed in Proposition \ref{prop: pi A projections}, still on every factor, all the corresponding $\pi_A$-projection of all the group elements appearing in this last expression of $\Omega_3$, that is for each of $g_{\underline{0}\underline{e_1}\underline{x}}^{-1}, g_{\underline{\infty}\underline{0}\underline{x}}^{-1} ,g_{\underline{\infty}\underline{0}\underline{x}}^{-1}$ and their left multiplication by $w_0^{-1}$ and $w_0^{-1}n^{-1}$. We can thus conclude that
\begin{eqnarray*}
\Omega_3(\underline{\infty},\underline{0},\underline{e_1},\underline{x})&=&{\alpha}\left(\left(a_{\frac{1}{\|x_i\|}}\right)_{i=1}^k,\left(a_{\|x_i\|}\right)_{i=1}^k,\left(a_{\frac{\|x_i\|}{1+\|x_i\|^2}}\right)_{i=1}^k\right)\\
&&-{\alpha}\left(\left(a_{\frac{1}{\|x_i-e_1\|}}\right)_{i=1}^k,\left(a_{\frac{\|x_i-e_1\|}{2}}\right)_{i=1}^k,\left(a_{\frac{\|x_i-e_1\|}{1+\|x_i\|^2}}\right)_{i=1}^k\right)\\
&&+{\alpha}\left(\left(a_{\frac{\|x_i\|}{\|x_i-e_1\|}}\right)_{i=1}^k,\left(a_{\frac{\|x_i-e_1\|}{2\|x_i\|}}\right)_{i=1}^k,\left(a_{\frac{\|x_i\|\|x_i-e_1\|}{1+\|x_i\|^2}}\right)_{i=1}^k\right) .
\end{eqnarray*}
We now replace the homogeneous cocycle ${\alpha}$ by its inhomogeneous variant to obtain 
\begin{eqnarray*}
\Omega_3(  \underline{\infty},\underline{0},\underline{e_1},\underline{x}  )&=&\widetilde{\alpha}\left( \left( a_{\|x_i\|^2}\right)_{i=1}^k, \left(a_{\frac{1}{1+\|x_i\|^2}}\right)_{i=1}^k\right)\\
&&-\widetilde{\alpha}\left(\left(a_{\frac{\|x_i-e_1\|^2}{2}}\right)_{i=1}^k,\left(a_{\frac{2}{1+\|x_i\|^2}}\right)_{i=1}^k\right)\\
&&+\widetilde{\alpha}\left(\left(a_{\frac{\|x_i-e_1\|^2}{2\|x_i\|^2}}\right)_{i=1}^k,\left(a_{\frac{2\|x_i\|^2}{1+\|x_i\|^2}}\right)_{i=1}^k\right) \\
&=&4\widetilde{\alpha}\left(\left(a_{\|x_i\|}\right)_{i=1}^k,\left(a_{\|x_i-e_1\|}\right)_{i=1}^k\right)  ,
\end{eqnarray*}
where for the last equality we have just used repeatedly that 
$$\begin{array}{lll}
\widetilde{\alpha}(ab,c)=\widetilde{\alpha}(a,c)+\widetilde{\alpha}(b,c), &\widetilde{\alpha}(a,bc)=\widetilde{\alpha}(a,b)+\widetilde{\alpha}(a,c),&\\
\widetilde{\alpha}(a^m,b^n)=nm\widetilde{\alpha}(a,b), &\widetilde{\alpha}(a,b)=-\widetilde{\alpha}(b,a) ,&\widetilde{\alpha}(a,a)=0,
\end{array}$$
for any $a,b,c\in A$ and $n,m\in \mathbb{R}$.

Finally observe that 
$$
b_i({\infty},0,e_1,x_i)=\|x_i\| \ . 
$$
and 
$$
{b}_i(0,e_1,x,{\infty})=\|x_i\|/\|x_i-e_1\|,
$$
so that we can rewrite $\Omega$ as 
$$\Omega_3(  \underline{\infty},\underline{0},\underline{e_1},\underline{x}  )=4\widetilde{\alpha}\left(\left(a_{b_i({\infty},0,e_1,x_i)}\right)_{i=1}^k,\left(a_{{b}_i(0,e_1,x,{\infty})}\right)_{i=1}^k\right) ,$$
which by $3$-transitivity, the $G$-invariance of $\Omega(e)$ and the $\mathrm{Isom}(\mathbb{H}^{n_i})$-invariance of the crossratios $b_i$ finishes the proof of the theorem.

\subsection{Proof of injectivity of the comparison map (Theorem \ref{injectivity for products}) for products of isometry groups of real hyperbolic space}\label{proof of thm injectivitiy}

\begin{lem}\label{lemma injectivity comparison actions} Let $G=\prod_{i=1}^k \mathrm{Isom}^+(\mathbb{H}^{n_i})$, where $n_i \geq 2$ for $1 \leq i \leq k$. Then Conjecture \ref{Conj injectivity boundary} is true in degree $3$, namely the comparison map
$$c_G:H^3_{m,b}(G\curvearrowright G/P)\longrightarrow H^3_m(G\curvearrowright G/P)$$
is injective.
\end{lem}

\begin{proof} If $f\in L^{\infty}((G/P)^4)$ is a $G$-invariant cocycle representing a cohomology class in $H^3_{m,b}(G\curvearrowright G/P)$  lying in the kernel of this comparison map, then it is the coboundary $f=\delta h$ of some not necessarily bounded $G$-invariant $h\in L^0((G/P)^3)$. Since the action of $G$ on triples of distinct points has a finite number of orbits (for $G=\mathrm{Isom}^+(\mathbb{H}^n)$ there is $1$ orbit when $n\geq 3$ and $2$ orbits when $n=2$), so any invariant cochain in degree $2$, and hence $h$, is bounded in either cases. 
\end{proof}

\begin{prop} \label{Omega3 unbounded hyperbolic} For any $\alpha\neq 0$, the cocycle 
$$s(\alpha_\mathfrak{a}):(x_0,\dots,x_3)\mapsto \alpha_{\mathfrak{a}}(\log(b(x_0,x_1,x_2,x_3)),\log(b(x_1,x_2,x_3,x_0))),$$
where $s$ is the section exhibited in Theorem \ref{Theorem 2 to 3 for prod of Hn} is unbounded.
\end{prop}

Before proving the unboundedness, we observe that this is sufficient to prove Conjecture \ref{Conj Gap} for $G$. Indeed any cohomology class in  $NH^3_m(G\curvearrowright G/P)$ is represented, by Theorem \ref{Theorem 2 to 3 for prod of Hn}, up to a coboundary, by a cocycle as in the proposition. By Lemma \ref{lemma injectivity comparison actions} any coboundary in degree $3$ is bounded, and the sum of a bounded and unbounded function is clearly unbounded. We thus have:

\begin{cor} \label{cor bdd deg 3} 
Conjecture \ref{Conj Gap} is true in degree $3$ for $G=\prod_{i=1}^k \mathrm{Isom}^+(\mathbb{H}^{n_i})$, where $n_i \geq 2$ for $1 \leq i \leq k$.
\end{cor}

Observe that Lemma \ref{lemma injectivity comparison actions} and Corollary \ref{cor bdd deg 3} prove Theorem  \ref{injectivity for products} for $G$. 

\begin{proof}[Proof of Proposition  \ref{Omega3 unbounded hyperbolic}] The group $(\wedge^2(\mathfrak{a}))^*$ has as basis the $2$ by $2$ determinants on the projections on pairs of factors in $\mathfrak{a}=\oplus_{i=1}^k \mathfrak{a}_i$, which viewed as inhomogeneous cocycles on $A$ take the form
$$
\alpha_{ij}:A \times A \rightarrow \bbR \ , \ \ \alpha_{ij}(a,a'):=\det \left( \begin{array}{cc} \log|a_i| & \log|a'_i| \\ \log|a_j| & \log|a'_j| \\ \end{array} \right) , 
$$
for $1\leq i<j\leq k$, $a=(a_1,\dots,a_k), a=(a'_1,\dots,a'_k)\in A$. We denote by $\Omega_{ij}:(G/P)^4\rightarrow \mathbb{R}$ the image of $\alpha_{ij}$ under the section $s$ of Theorem \ref{Theorem 2 to 3 for prod of Hn}. In particular, 
$$\Omega_{ij}(\underline{\infty},\underline{0},\underline{e_1},x)=\det \left( \begin{array}{cc} \log\|x_i\| & \log\|e_1-x_i\| \\ \log\|x_j\| & \log\|e_1-x_j\| \\ \end{array} \right) , $$
for any $x=(x_1,\dots,x_k)\in G/P=\Pi G_i/P_i$. Any class $\Omega$ in the image of the section $s$ of Theorem \ref{Theorem 2 to 3 for prod of Hn}  is a linear combination $\Omega=\sum_{i<j}t_{ij}\Omega_{ij}$. We assume that $\Omega\neq 0$ so that at least one of the coefficients $t_{ij}\neq 0$. By symmetry we can suppose that $t_{12}\neq 0$. We claim that there exists a subset of positive measure of $G_2/P_2\times \dots \times G_k/P_k$ satisfying 
$$\sum_{j=2}^k t_{1j} \log \|x_j\| >0.$$
Indeed, just choose $x_j\in G_j/P_j=\mathbb{R}^{n_j-1}$ such that $t_{12} \log \|x_2\| > 0$ for $j=2$, and $t_{1j} \log \|x_j\| \geq 0$ for $3\leq j\leq k$. This is simply achieved by choosing $x_j$ such that $$\begin{array}{ll}
\| x_j\| >1&\mathrm{ if \ }t_{1j}>0,\\
\mathrm{arbitrary}&\mathrm{if \ } t_{1j}=0,\\
\| x_j\| <1&\mathrm{ if \ }t_{1j}<0.
\end{array} $$
Now for any such $(x_2,\dots,x_k)$, consider $x=(x_1,\dots,x_k)\in G/P$ with $x_1\in G_1/P_1=\mathbb{R}^{n_1-1}$. We have
\begin{align*}
\Omega(\underline{\infty},\underline{0},\underline{e_1},x)=&\sum_{j=2}^k t_{1j} \det \left( \begin{array}{cc} \log\|x_1\| & \log\|e_1-x_1\| \\ \log\|x_j\| & \log\|e_1-x_j\| \\ \end{array} \right)\\
& + \mathrm{a \ constant \ depending \ on \ } x_2,\dots,x_k\\
=& \log\|x_1\| \underbrace{(\sum_{j=2}^k t_{1j} \log \| e_1-x_j\|)}_{C_1} -\log\|e_1-x_1\| \underbrace{(\sum_{j=2}^k t_{1j} \log \| x_j\|)}_{C_2}+C_3\\
=&\log\|x_1\| C_1-\log\|e_1-x_1\| C_2+C_3,
\end{align*}
where $C_1,C_2,C_3$ depend solely on $x_2,\dots,x_k$ which have been chosen so that $C_2\neq 0$. Letting now $x_1$ tend to $e_1$ shows that $\Omega$ is unbounded.
\end{proof}

\subsection{Proof of Theorem \ref{Theorem 2 to 4 for prod of Hn}}\label{proof of thm 3}

As before, let $\alpha_\mathfrak{a} \in \wedge^2\mathfrak{a}^\ast$ be an alternating form, $\widetilde{\alpha}:A^2\rightarrow \mathbb{R}$ be the corresponding inhomogenous $2$-cocycle 
$$\widetilde{\alpha}(a_1,a_2):= \alpha_\mathfrak{a}(\log a_1,\log a_2)$$
and $\alpha:A^3\rightarrow \mathbb{R}$ be the homogenous $A$-invariant cocycle 
$$\alpha(a_0,a_1,a_2)=\widetilde{\alpha}(a_0^{-1}a_1,a_1^{-1}a_2).$$
The cocycle
$$\alpha_G:G^3\longrightarrow \mathbb{R}$$
defined by
$$\alpha_G(g_0,g_1,g_2)=\alpha(\pi_A(g_0),\pi_A(g_1),\pi_A(g_2))$$
again corresponds to $\alpha_\mathfrak{a}$ in $H^2(C^{2,1},d^\uparrow)\cong H^2_m(P)\cong H^2_m(A)\cong (\wedge^2 \mathfrak{a})^*$. According to Lemma \ref{Injection} , we need to follow $\overline{\overline{\alpha_G}}$ through the maps
\begin{equation}\label{eq:diagram:chasing}
\xymatrix{
\overline{\overline{\alpha_G}} \in C^{2,2} \ar[r] & d^\rightarrow \overline{\overline{\alpha_G}}=d^\uparrow \beta \in C^{2,3} & &\\
& h^{2,3}(d^\rightarrow\alpha_G)=:\beta \in C^{1,3} \ar[u] \ar[r] & d^\rightarrow \beta=d^\uparrow \omega \in C^{1,4} &\\
& & h^{1,4}(d^\rightarrow \beta)=:\omega \in C^{0,4} \ar[u] \ar[r] & d^\rightarrow \omega =:\Omega_4\in C^{0,5} ,
}
\end{equation}
and evaluate $\Omega_4$ on $e$. As before we are going to subdivide the diagram chase into several steps. 

\subsection*{Computation of $\beta$} By Lemma \ref{d 2 to 3} we know that $d^\rightarrow \alpha_G$ is as an element of $L^0(G^3)=C^{2,3}$ given by 
\begin{align*}
d^\rightarrow \alpha_G(g_0,g_1,g_2)=&-\alpha_G(nw_0^{-1}g_0,nw_0^{-1}g_1,nw_0^{-1}g_2)+\alpha_G(n^{-1}g_0,n^{-1}g_1,n^{-1}g_2)\\
&-\alpha_G(g_0,g_1,g_2)\\
=&-\alpha_G(w_0^{-1}g_0,w_0^{-1}g_1,w_0^{-1}g_2) ,
\end{align*}
where we used the $P$-invariance of $\alpha_G$. Recall that the signs in the above formula are due to the weight used to define $d^\rightarrow$. 

Observe that the homotopy $h^{p+1,3}:C^{p,3}\rightarrow C^{p-1,3}$ is as a map from $L^0(G^{p+1})=C^{p,3}\rightarrow C^{p-1,3}=L^0(G^p)$ simply given by 
$$h^{p+1,3}(\gamma)(g_0,\dots,g_{p-1})=\gamma(e,g_0,\dots,g_{p-1}),$$
for $\gamma\in G^{p+1}$. Applying this to $\gamma=d^\rightarrow \alpha$ leads to
\begin{align*}
\beta(g_0,g_1)=&h^{2,3}(d^\rightarrow \alpha_G)(g_0,g_1)\\
=&(d^\rightarrow \alpha_G)(e,g_0,g_1)\\
=&-\alpha_G(e,w_0^{-1}g_0,w_0^{-1}g_1) .
\end{align*}



\subsubsection*{Computation of $\omega$} Now we want to compute the differential $d^\rightarrow \beta\in C^{1,4}$. By $3$-transitivity any $4$-tuple of points in $G/P$ is in the orbit of  $\underline{\infty},\underline{0},\underline{e_1},x$, for some $x\in G/P$. Since $d^\rightarrow \beta$ is $G$-invariant it will thus be sufficient to evaluate it on such $4$-tuples. As above, for any triple of distinct points $x,y,z \in G/P$, we let $g_{xyz}$ be an element in $G$ such that $g_{xyz}.\underline{\infty}=x,g_{xyz}.\underline{0}=y,g_{xyz}.\underline{e_1}=z$. We can thus express $d^\rightarrow \beta$ as
\begin{align*}
d^\rightarrow \beta(g_0,g_1)(\underline{\infty},\underline{0},\underline{e_1},x)=&\beta(g_0,g_1)(\underline{0},\underline{e_1},x)-\beta(g_0,g_1)(\underline{\infty},\underline{e_1},x)\\
&+\beta(g_0,g_1)(\underline{\infty},\underline{0},x)+\beta(g_0,g_1)(\underline{\infty},\underline{0},\underline{e_1})\\
=&\beta(g_{\underline{0}\underline{e_1}x}^{-1}g_0,g_{\underline{0}\underline{e_1}x}^{-1}g_1)-\beta(g_{\underline{\infty} \underline{e_1}x}^{-1}g_0,g_{\underline{\infty} \underline{e_1}x}^{-1}g_1)\\
&+\beta(g_{\underline{\infty}\underline{0}x}^{-1}g_0,g_{\underline{\infty}\underline{0}x}^{-1}g_1)+\beta(g_0,g_1)  .
\end{align*}

Finally we obtain
\begin{align*}
\omega(g)(\underline{\infty},\underline{0},\underline{e_1},x):=&h^{1,4}(d^\rightarrow \beta)(g)(\underline{\infty},\underline{0},\underline{e_1},x)\\
=&d^\rightarrow \beta(e,g)(\underline{\infty},\underline{0},\underline{e_1},x)\\
=&\beta(g_{\underline{0}\underline{e_1}x}^{-1},g_{\underline{0}\underline{e_1}x}^{-1}g)-\beta(g_{\underline{\infty} \underline{e_1}x}^{-1},g_{\underline{\infty}\underline{ e_1}x}^{-1}g)+\beta(g_{\underline{\infty}\underline{0}x}^{-1},g_{\underline{\infty}\underline{0}x}^{-1}g)  ,
\end{align*}
where we have omitted the last term $\beta(e,g)$ since it vanishes. 

\subsubsection*{Computation of $d^\rightarrow \omega$} We are finally ready to compute our desired cocycle $\Omega_4=d^\rightarrow \omega(e)$. Before starting, recall that an oriented $5$-tuple of distinct points is in the $G$-orbit of $\underline{\infty},\underline{0},\underline{e_1},x,y$, for some $x\neq y\in G/P\setminus\{\underline{\infty},\underline{0},\underline{e_1}\}$. For such $x,y$, we set
$$
\lambda(x,y):=g_{\underline{\infty}\underline{0}x}^{-1}.y , \ \ \mu(x,y):=g_{\underline{\infty} \underline{e_1}x}^{-1}.y , \ \ \nu(x,y):=g_{\underline{0}\underline{e_1} x}^{-1}. y ,
$$
where $g_{\underline{\infty}\underline{0}x}$, $g_{\underline{\infty} \underline{e_1}x}$ and $g_{\underline{0}\underline{e_1} x}$ are chosen as above. We have
\begin{align*}
\Omega_4(\underline{\infty},\underline{0},\underline{e_1},x,y)=&d^\rightarrow \omega(e)(\underline{\infty},\underline{0},\underline{e_1},x,y)\\
=&-\omega(e)(\underline{0},\underline{e_1},x,y)+\omega(e)(\underline{\infty},\underline{e_1},x,y)-\omega(e)(\underline{\infty},\underline{0},x,y)\\
&+\omega(e)(\underline{\infty},\underline{0},\underline{e_1},y)-\omega(e)(\underline{\infty},\underline{0},\underline{e_1},x)\\
=&-\omega(g_{\underline{0}\underline{e_1}x}^{-1})(\underline{\infty},\underline{0},\underline{e_1},\nu(x,y))+\omega(g_{\underline{\infty} \underline{e_1}x}^{-1})(\underline{\infty},\underline{0},\underline{e_1},\mu(x,y))\\
&-\omega(g_{\underline{\infty}\underline{0}x}^{-1})(\underline{\infty},\underline{0},\underline{e_1},\lambda(x,y)) ,
\end{align*}
where we used $G$-invariance of $\omega$ for the last equality and the fact that both the fourth and the fifth summands vanish. We will now compute the three remaining summands separately. 

\subsubsection*{Computation of $\omega(g_{\underline{0}\underline{e_1}x}^{-1})(\underline{\infty},\underline{0},\underline{e_1},\nu(x,y))$}  

We have that 
\begin{align*}
&\omega(g_{\underline{0}\underline{e_1}x}^{-1})(\underline{\infty},\underline{0},\underline{e_1},\nu(x,y))=\\
=&\beta(g_{0e_1\nu(x,y)}^{-1},g_{0e_1\nu(x,y)}^{-1}g_{ 0e_1 x}^{-1})-\beta(g_{\infty e_1\nu(x,y)}^{-1},g_{\infty e_1\nu(x,y)}^{-1}g_{0e_1 x}^{-1})
+\beta(g_{\infty 0\nu(x,y)}^{-1},g_{\infty 0\nu(x,y)}^{-1}g_{0 e_1x}^{-1})\\
=&-\alpha_G(e,w_0^{-1}g_{0e_1\nu(x,y)}^{-1},w_0^{-1}g_{0e_1\nu(x,y)}^{-1}g_{0e_1 x}^{-1})+\alpha_G(e,w_0^{-1}g_{\infty e_1\nu(x,y)}^{-1},w_0^{-1}g_{\infty e_1\nu(x,y)}^{-1}g_{0e_1 x}^{-1})\\
&-\alpha_G(e,w_0^{-1}g_{\infty 0\nu(x,y)}^{-1},w_0^{-1}g_{\infty 0\nu(x,y)}^{-1}g_{0e_1 x}^{-1}) .
\end{align*}

The six different projections on $A$ have been exhibited in Proposition \ref{piA for nu} coordinatewise so that we simply obtain  
\begin{align*}
\omega(g^{-1}_{0e_1 x})(\infty,0,e_1,\nu(x,y))=&-\widetilde{\alpha}\left( \left(  a_{\frac{\|y_i-x_i\|}{2\|x_i\|\|y_i-e_1\|}}\right)_{i=1}^k
,\left( a_{\frac{2\|x_i\|\|x_i-e_1\|}{\|y_i-x_i\|(1+\|x_i\|^2)}}\right)_{i=1}^k
\right)\\
&+\widetilde{\alpha}\left(\left( a_{\frac{\|y_i-x_i\|}{2\|y_i\|\|x_i-e_1\|}}\right)_{i=1}^k
,\left( a_{\frac{2\|x_i\|\|x_i-e_1\|}{\|y_i-x_i\|(1+\|x_i\|^2)}}\right)_{i=1}^k
\right)\\
&-\widetilde{\alpha}\left(\left( a_{\frac{\|x_i\|\|y_i-e_1\|}{\|y_i\|\|x_i-e_1\|}}\right)_{i=1}^k
,\left( a_{\frac{\|x_i-e_1\|}{2\|x_i\|}}\right)_{i=1}^k
\right)\\
=&-\widetilde{\alpha}\left(\left( a_{\frac{\|x_i\|\|y_i-e_1\|}{\|y_i\|\|x_i-e_1\|}}\right)_{i=1}^k,\left( a_{\frac{\|x_i\|^2+1}{4\|x_i\|^2}}\right)_{i=1}^k \right).
\end{align*}

\subsubsection*{Computation of $\omega(g_{\underline{\infty} \underline{e_1}x}^{-1})(\underline{\infty},\underline{0},\underline{e_1},\mu(x,y)))$} We have that
\begin{align*}
\omega(g_{\underline{\infty} \underline{e_1}x}^{-1})(\underline{\infty},\underline{0},\underline{e_1},\mu(x,y))
=&\beta(g_{\underline{0} \underline{e_1}\mu(x,y)}^{-1},g_{\underline{0} \underline{e_1}\mu(x,y)}^{-1}g_{\underline{\infty}  \underline{e_1} x}^{-1})-\beta(g_{\underline{\infty}  \underline{e_1}\mu(x,y)}^{-1},g_{\underline{\infty}  \underline{e_1}\mu(x,y)}^{-1}g_{\infty  \underline{e_1} x}^{-1}) \\
&+\beta(g_{\underline{\infty} \underline{0}\mu(x,y)}^{-1},g_{\underline{\infty} \underline{0}\mu(x,y)}^{-1}g_{\underline{\infty}  \underline{e_1} x}^{-1})\\
=&-\alpha_G(e,w_0^{-1}g_{\underline{0} \underline{e_1}\mu(x,y)}^{-1},w_0^{-1}g_{\underline{0}e_1\mu(x,y)}^{-1}g_{\underline{\infty}  \underline{e_1} x}^{-1})\\
&+\alpha_G(e,w_0^{-1}g_{\underline{\infty}  \underline{e_1}\mu(x,y)}^{-1},w_0^{-1}g_{\underline{\infty} e_1\mu(x,y)}^{-1}g_{\underline{\infty}  \underline{e_1} x}^{-1})\\
&-\alpha_G(e,w_0^{-1}g_{\underline{\infty} \underline{0}\mu(x,y)}^{-1},w_0^{-1}g_{\underline{\infty}\underline{0}\mu(x,y)}^{-1}g_{\underline{\infty} \underline{e_1} x}^{-1}).
\end{align*}
We have already computed these six different projections on $A$ coordinatewise in Proposition \ref{piA for mu} so that we can just conclude that 
\begin{align*}
&\omega(g_{\underline{\infty} \underline{e_1}x}^{-1})(\underline{\infty},\underline{0},\underline{e_1},\mu(x,y))\\
=&-\widetilde{\alpha}\left(\left( a_{\frac{\|y_i-x_i\|}{2\|y_i-e_1\|}}\right)_{i=1}^k,\left( a_{\frac{2\|x_i-e_1\|}{\|x_i\|^2+1}}\right)_{i=1}^k\right)+\widetilde{\alpha}\left(\left( a_{\frac{\|y_i-x_i\|}{2\|x_i-e_1\|}}\right)_{i=1}^k,\left( a_{\frac{2\|x_i-e_1\|}{\|x_i\|^2+1}}\right)_{i=1}^k\right)\\
&-\widetilde{\alpha}\left(\left( a_{\frac{\|y_i-e_1\|}{\|x_i-e_1\|}}\right)_{i=1}^k,\left( a_{\frac{\|x_i-e_1\|}{2}}\right)_{i=1}^k\right) \\
=&-\widetilde{\alpha}\left(\left( a_{\frac{\|y_i-e_1\|}{\|x_i-e_1\|}}\right)_{i=1}^k,\left( a_{\frac{\|x_i\|^2+1}{4}}\right)_{i=1}^k\right) .
\end{align*}

\subsubsection*{Computation of $\omega(g_{\underline{\infty}\underline{0}x}^{-1})(\underline{\infty},\underline{0},\underline{e_1},\lambda(x,y))$}
We have that
\begin{align*}
&\omega(g_{\underline{\infty}\underline{0}x}^{-1})(\underline{\infty},\underline{0},\underline{e_1},\lambda(x,y))\\
=&\beta(g_{0e_1\lambda(x,y)}^{-1},g_{0e_1\lambda(x,y)}^{-1}g_{\infty 0 x}^{-1})-\beta(g_{\infty e_1\lambda(x,y)}^{-1},g_{\infty e_1\lambda(x,y)}^{-1}g_{\infty 0 x}^{-1})
+\beta(g_{\infty 0\lambda(x,y)}^{-1},g_{\infty 0\lambda(x,y)}^{-1}g_{\infty 0 x}^{-1})\\
=&-\alpha_G(e,w_0^{-1}g_{0e_1\lambda(x,y)}^{-1},w_0^{-1}g_{0e_1\lambda(x,y)}^{-1}g_{\infty0 x}^{-1})+\alpha_G(e,w_0^{-1}g_{\infty e_1\lambda(x,y)}^{-1},w_0^{-1}g_{\infty e_1\lambda(x,y)}^{-1}g_{\infty 0 x}^{-1})\\
&-\alpha_G(e,w_0^{-1}g_{\infty 0\lambda(x,y)}^{-1},w_0^{-1}g_{\infty 0\lambda(x,y)}^{-1}g_{\infty 0 x}^{-1})  .
\end{align*}
We have already recorded these six different projections on $A$ coordinatewise in Proposition \ref{piA for lambda} so that we can just conclude that

\begin{align*}
\omega(g_{\underline{\infty}\underline{0}x}^{-1})(\underline{\infty},\underline{0},\underline{e_1},\lambda(x,y))=&-\widetilde{\alpha}\left( \left( a_{\frac{\|y_i-x_i\|}{2\|y_i\|}}\right)_{i=1}^k
,\left( a_{\frac{2\|x_i\|}{1+\|x_i\|^2}}\right)_{i=1}^k
\right)\\
&+\widetilde{\alpha}\left( \left( a_{\frac{\|y_i-x_i\|}{2\|x_i\|}}\right)_{i=1}^k
, \left( a_{\frac{2\|x_i\|}{\|x_i\|^2+1}} \right)_{i=1}^k
\right)\\
&-\widetilde{\alpha}\left(\left( a_{\frac{\|y_i\|}{\|x_i\|}}\right)_{i=1}^k
,\left(  a_{\|x_i\|}\right)_{i=1}^k
\right)\\
=&-\widetilde{\alpha}\left(\left( a_{\frac{\|y_i\|}{\|x_i\|}}\right)_{i=1}^k
,\left( a_{\frac{\|x_i\|^2+1}{2}}\right)_{i=1}^k
\right)  .
\end{align*}

\subsection*{Conclusion.}

We can finally compute $\Omega_4$ putting everything together:
\begin{align*}
\Omega_4(\underline{\infty},\underline{0},\underline{e_1},x,y)=& 
\widetilde{\alpha}\left(\left( a_{\frac{\|x_i\|\|y_i-e_1\|}{\|y_i\|\|x_i-e_1\|}}\right)_{i=1}^k,\left( a_{\frac{\|x_i\|^2+1}{4\|x_i\|^2}}\right)_{i=1}^k \right) \\
&-\widetilde{\alpha}\left(\left( a_{\frac{\|x_i\|\|y_i-e_1\|}{\|y_i\|\|x_i-e_1\|}}\right)_{i=1}^k,\left( a_{\frac{\|x_i\|^2+1}{4\|x_i\|^2}}\right)_{i=1}^k \right)\\
&+\widetilde{\alpha}\left(\left( a_{\frac{\|y_i\|}{\|x_i\|}}\right)_{i=1}^k,\left( a_{\frac{\|x_i\|^2+1}{2}}\right)_{i=1}^k\right)\\
=&\widetilde{\alpha}\left(\left( a_{\|x_i\|}\right)_{i=1}^k
,\left(a_{\frac{\|x_i\|\|y_i-e_1\|}{\|y_i\|\|x_i-e_1\|}}\right)_{i=1}^k
\right)\\
&+\widetilde{\alpha} \left( \left( a_\frac{\|y_i\|}{\|x_i\|}\right)_{i=1}^k, (a_2)_{i=1}^k \right) .
\end{align*}
Now the second summand of this last expression is the coboundary of the function $\beta:(G/P)^4\rightarrow \mathbb{R}$ defined by
$$\beta(\underline{\infty},\underline{0},\underline{e_1},x)=\widetilde{\alpha}\left( \left( a_{\|x_i\|}\right)_{i=1}^k, (a_2)_{i=1}^k \right)$$
and the first summand is exactly the expression claimed in Theorem \ref{Theorem 2 to 4 for prod of Hn}.

\section{The spectral sequence for $\mathrm{SL}(3,\mathbb{K})$ for $\mathbb{K}=\mathbb{R},\mathbb{C}$}\label{SL3R}

\subsection{Configurations of flags in general position}\label{sec general position}\label{SL3R1}

Let $\mathbb{K}$ be either the real or the complex field. For $G=\mathrm{SL}(3,\mathbb{K})$, we take as minimal parabolic subgroup $P$ the subgroup of upper triangular matrices with entries in $\mathbb{K}$. The quotient $G/P$ is then naturally identified with the space of complete flags $\mathcal{FL}(3,\mathbb{K})$. Recall that a \emph{complete flag} $F\in \mathcal{FL}(3,\mathbb{K})$ in $\mathbb{K}^3$ is a sequence of nested subspaces
$$
F: \ (0)=F^0 \subset F^1 \subset F^2 \subset F^3=\mathbb{K}^3 ,
$$
where each linear subspace $F^i$ has dimension $\dim_{\mathbb{K}} F^i=i$. In this case, the flag is completely determined by the $1$ and $2$ dimensional subspaces. For this reason, we can alternatively denote a complete flag $F$ by a pair $(p,\ell)$, where $p \in P^2(\mathbb{K})$ is a point and $\ell \subset P^2(\mathbb{K})$ is a line passing through $p$. The subgroup $P$ is the stabilizer of the  \emph{canonical flag} $F_{can}$:
$$
F_{can}: \ (0) \subset \langle e_1\rangle  \subset \langle e_1,e_2\rangle  \subset \mathbb{K}^3 ,
$$
where $\{e_1,e_2,e_3\}$ denotes the canonical basis of $\mathbb{K}^3$. 

The natural action of $\mathrm{SL}(3,\mathbb{K})$ on $\mathcal{FL}(3,\mathbb{K})$ induces a diagonal action on the product $\mathcal{FL}(3,\mathbb{K})^{d+1}$. The space of \emph{configurations} of $(d+1)$-tuples of flags is the quotient 
$$
\mathcal{C}_{d+1}(\mathcal{FL}(3,\mathbb{K})):=\mathcal{FL}(3,\mathbb{K})^{d+1}/\mathrm{SL}(3,\mathbb{K}) .
$$
Given a $(d+1)$-tuple $(F_0,\cdots,F_d)$ of flags in $\mathcal{FL}(3,\mathbb{K})$, we denote its configuration class by $[F_0,\cdots,F_d]$. 

The cocycles we are going to define will never be defined everywhere but only on the measurable dense subset of flags in general position. A $(d+1)$-tuple of flags $F_0,\dots,F_d$ is in general position if, roughly speaking, the dimension of any possible intersection or subspace generated by the spaces from the tuple matches the expected dimension. In dimension $3$, we can simply (and equivalently) define a $(d+1)$-tuple of flags $F_0.\dots,F_d$, with $F_i=(p_i,\ell_i)$, to be in \emph{general position}  if 
\begin{itemize}
\item  $p_{i_1},p_{i_2},p_{i_3}$ are not aligned whenever $|\{i_1,i_2,i_3\}|=3$, 
\item  $\ell_{i_1},\ell_{i_2},\ell_{i_3}$ do not intersect in a unique point whenever $|\{i_1,i_2,i_3\}|=3$, 
\item $p_i\notin \ell_j$ whenever $i\neq j$. 
\end{itemize}
This is precisely the notion of very generic configurations from \cite{FW17}. Since the condition of general position is invariant along the $\mathrm{SL}(3,\mathbb{K})$-orbits, it makes sense to speak about configurations of tuples of flags in general position. We denote this space by $\mathcal{C}^{gen}_{\ast}(\mathcal{FL}(3,\mathbb{K}))$. 

We follow Falbel and Wang to introduce coordinates on the spaces of triples and $4$-tuples of flags in general position and refer to \cite{FW17}  for more details. In \cite{FW17} the definitions are given for tuples of complex flags, but it should be clear that the same definitions work for tuples of real flags as well. 

\begin{dfn} \label{dfn triple ratio}
Let $F_0,F_1,F_2 \in \mathcal{FL}(3,\mathbb{K})$ be a triple of flags in general position, where $F_i=(p_i,\ell_i)$ for $i=0,1,2$. We define the \emph{triple ratio} associated to them as 
$$
\tau(F_0,F_1,F_2):=-[\ell_0 \cap \ell_1, \ell_1 \cap \ell_2,p_1,\ell_1 \cap p_0 \cdot p_2]_{\ell_1} \in \mathbb{K}^\ast \setminus \{-1\} , 
$$
where $\ell_i \cap \ell_j$ is the intersection point between the lines and $p_0 \cdot p_2$ is the line passing through $p_0$ and $p_2$. The notation $[\cdot , \cdot  ,  \cdot , \cdot]_{\ell_1}$ refers to the usual cross ratio computed on the line $\ell_1$. (See Figure \ref{Figure triple ratio}.)
\end{dfn}

In view of  \cite[Lemma 3.5]{FW17}, this is equivalent to the original definition given by Falbel and Wang. The triple ratio $\tau$ of $(F_0,F_1,F_2)$ changes equivariantly with respect the action of a permutation $\sigma \in S_3$: more precisely, the triple ratio of $(F_{\sigma(0)},F_{\sigma(1)},F_{\sigma(2)})$ is given by $\tau^{\varepsilon(\sigma)}$, where $\varepsilon(\sigma)$ is the sign of the permutation $\sigma$. 

The triple ratio remains constant along $G$-orbits, hence it descends to a well-defined numerical invariant of configuration classes. By \cite[Proposition 3.1]{FW17} it determines an identification between the space of configurations in general position $\mathcal{C}^{gen}_3(\mathcal{FL}(3,\mathbb{K}))$ and  $\mathbb{K}^\ast \setminus \{-1\}$. Indeed, since $\mathrm{SL}(3,\mathbb{K})$ acts transitively on the triples given by two generic flags and one generic point in $P^2(\mathbb{K})$, any triple of flags lies in the same orbit as a triple $(F_0,F_1,F_2)$, where 
\begin{equation}\label{transitivity}
F_{can}=F_0=\langle e_1 \rangle  \subset \langle e_1,e_2 \rangle \subset \mathbb{K}^3 ,
\end{equation}
$$
F_1=\langle e_3 \rangle \subset \langle e_3,e_2 \rangle \subset \mathbb{K}^3 ,
$$
$$(F_2)^1=\langle e_1+e_2+e_3 \rangle .$$
Now for any $\tau \in \mathbb{K}^\ast \setminus \{-1\}$ there exists a unique way to complete $\langle e_1+e_2+e_3 \rangle$ to a flag $F_2$ such that $\tau(F_0,F_1,F_2)=\tau$. More precisely, the flag $F_2$ is given by 
\begin{equation}\label{Standard triple flags}
F_2=\langle e_1+e_2+e_3 \rangle \subset \langle e_1+e_2+e_3,(\tau+1)e_1 +\tau e_2 \rangle \subset \mathbb{K}^3 .
\end{equation}


\begin{dfn}\label{dfn cross ratios tuple}
Let $F_0,F_1,F_2,F_3 \in \mathcal{FL}(3,\mathbb{K})$ be a $4$-tuple of flags in general position, where $F_i=(p_i,\ell_i)$ for $i=0,\cdots,3$. Let $(i,j,s,t)$ be an even permutation of $(0,1,2,3)$. The $(i,j)$-\emph{cross ratio} is defined as
$$
z_{ij}:=[\ell_i,p_i \cdot p_j,p_i \cdot p_s,p_i \cdot p_t]_{p_i} \ ,
$$
where $p_i \cdot p_j$ refers to the line passing through $p_i$ and $p_j$. The cross ratio $[\cdot,\cdot,\cdot,\cdot]_{p_i}$ is computed on the line parametrizing all the lines passing through $p_i$. (See Figure \ref{Figure cross ratio}.)
\end{dfn}

\begin{figure}[htb]
    \begin{minipage}[t]{.45\textwidth}
        \centering
        \includegraphics[width=1.1\textwidth]{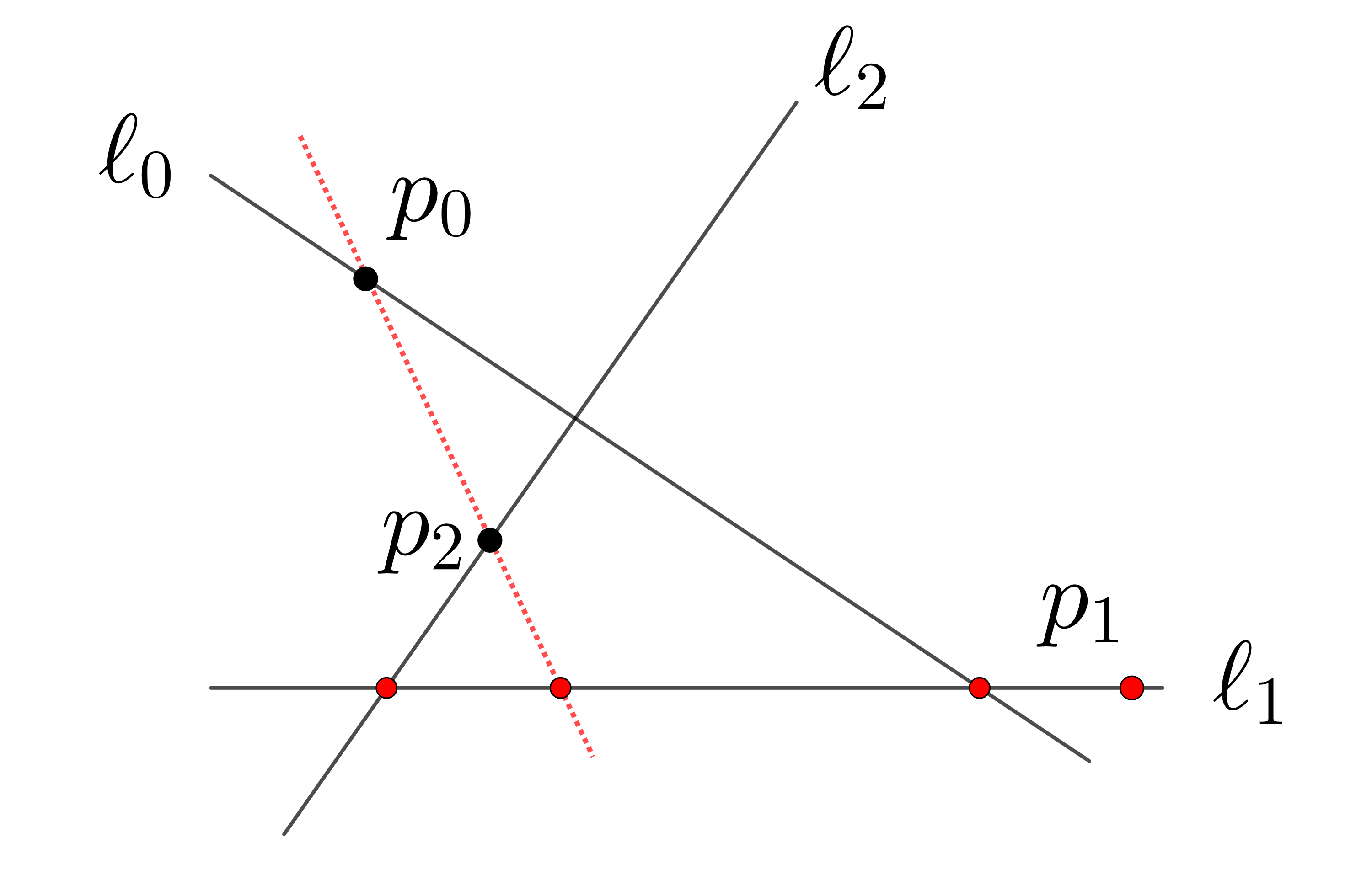}
        \caption{Triple ratio of a triple of flags}\label{Figure triple ratio}
    \end{minipage}
    \hfill
    \begin{minipage}[t]{.45\textwidth}
        \centering
        \includegraphics[width=1.1\textwidth]{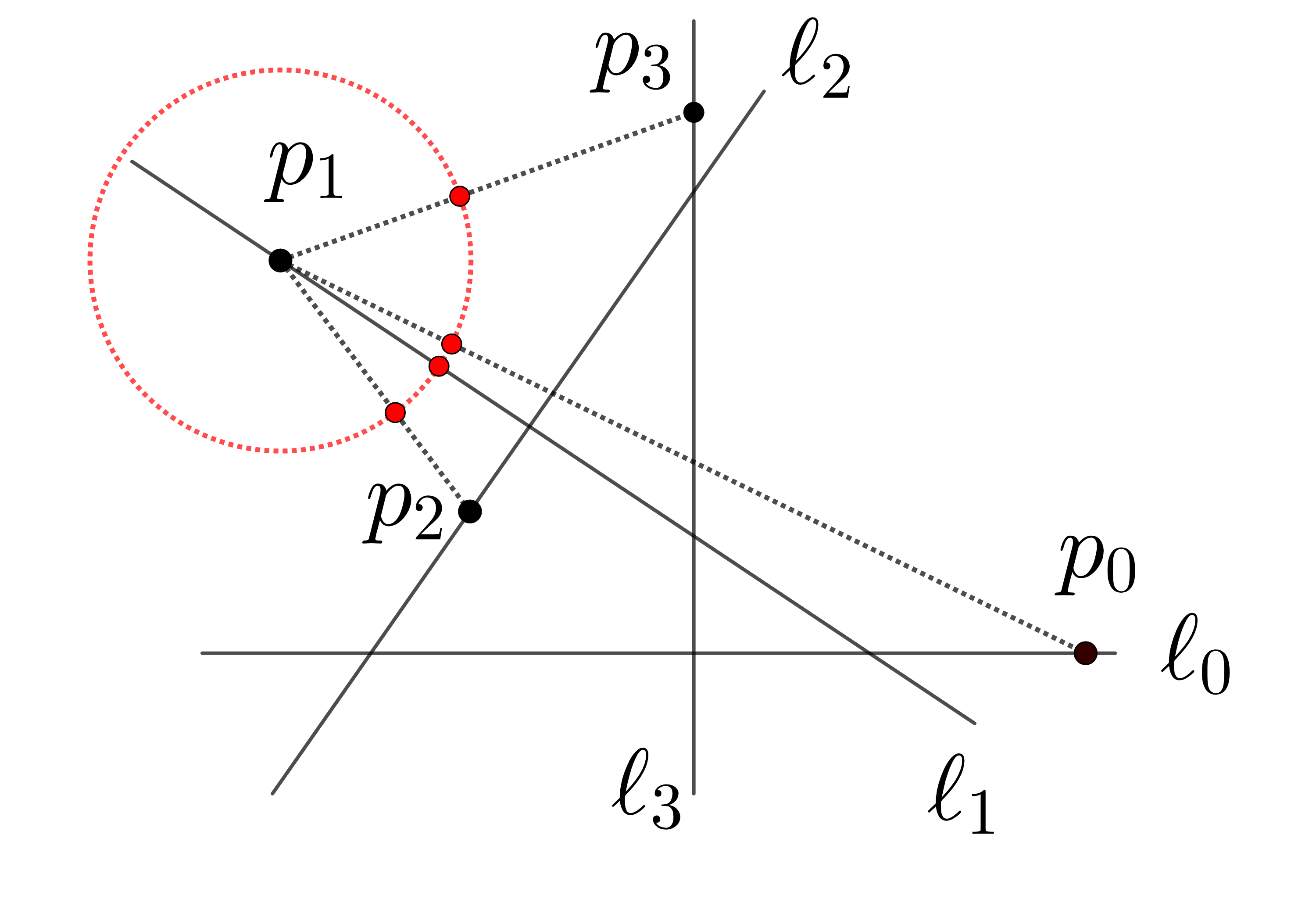}
        \caption{Cross ratio around the point $p_1$}\label{Figure cross ratio}
    \end{minipage}  
\end{figure}

The family of cross ratios given by Definition \ref{dfn cross ratios tuple} is constant along the $G$-orbits, hence it descends naturally to configuration classes. By \cite[Proposition 3.1]{FW17} the choice of cross ratios $(z_{01},z_{10},z_{23},z_{32})$ determines an isomorphism between the space of configurations in general position $\mathcal{C}^{gen}_4(\mathcal{FL}(3,\mathbb{K}))$ and $(\mathbb{K}^\ast \setminus \{1 \})^4$. Indeed, any $4$-tuple of generic flags lies in the same orbit as $(F_0,F_1,F_2,F_3)$, where $F_0,F_1,(F_2)^1$ are as in (\ref{transitivity}). Now for any $a,b,c,d \in \mathbb{K}^\ast \setminus \{ 1 \}$ there exist a unique way to complete $(F_2)^1$ to a flag and a unique flag $F_3$ such that the list of cross ratios $(z_{01},z_{10},z_{23},z_{32})$ is given by $(a,b,c,d)$. Indeed, it is sufficient to take
\begin{equation}\label{Standard four flags}
F_0=\langle e_1 \rangle  \subset \langle e_1,e_2 \rangle \subset \mathbb{K}^3  ,
\end{equation}
$$
F_1=\langle e_3 \rangle \subset \langle e_3,e_2 \rangle \subset \mathbb{K}^3  ,
$$
$$
F_2=\langle e_1+e_2+e_3 \rangle \subset \langle e_1+e_2+e_3,(\tau_{012}+1)e_1 +\tau_{012} e_2 \rangle \subset \mathbb{K}^3 ,
$$
$$
F_3=\langle ab \ e_1+a \ e_2+e_3 \rangle \subset \langle ab \ e_1+ a \ e_2+e_3,(\tau_{013}+1)b \ e_1+\tau_{013}e_2 \rangle \subset \mathbb{K}^3 .
$$
Here $\tau_{ijk}$ is the triple ratio of the configurations $[F_i,F_j,F_k]$. Notice that we can express such triple ratios in terms of $a,b,c,d$. By Falbel and Wang \cite[Equation 3.4.2]{FW17} we have
\begin{equation}\label{eq triple in terms cross}
\tau_{123}=\frac{d(c-1)}{b(d-1)}  , \ \ \tau_{023}=\frac{a(c-1)}{c(d-1)}  , 
\end{equation}
$$
\tau_{013}=\frac{b(a-1)}{d(b-1)}  , \ \ \tau_{012}=\frac{c(a-1)}{a(b-1)} .
$$

\subsection{The invariant cohomology of the maximal torus} \label{SL3R2}

In virtue of Theorem \ref{Thm Monod} we need first to compute the invariant cohomology of the (dual) Lie algebra of a maximal abelian torus $A$ in $G$. We can choose $A$ as the set of diagonal matrices with positive real entries, namely
$$
A=\{ \ \textup{diag}(\lambda,\mu,\nu) \ | \ \lambda,\mu,\nu>0, \ \lambda \mu \nu=1 \ \}  ,
$$
with associated Lie algebra given by the space of diagonal traceless matrices with real entries, that is
$$
\mathfrak{a}:= \{ \ \textup{diag}(x,y,z) \ | \ x+y+z=0 \ \}  .
$$

\begin{lem}\label{lem invariant SL3}
If $w_0$ represents the longest element in the Weyl group of $A$, the $w_0$-invariant cohomology of the dual algebra $\mathfrak{a}^\ast$ is given by
$$
\dim (\wedge^\ell \mathfrak{a}^\ast)^{w_0}=
\begin{cases*}
1, \ \ \ & \textup{for $\ell=1$,} \\
0, \ \ \ & \textup{for $\ell \geq 2$.}  
\end{cases*}
$$
\end{lem}

\begin{proof}
Since $\mathfrak{a}$ has dimension equal to $2$, it is sufficient to check the cases when $\ell$ is equal to either $1$ or $2$. 

We start by noticing that the Weyl group $W$ of $G$ is isomorphic to
$$
W \cong \mathrm{Sym}(3) \cong D_6 \ .
$$
The longest element $w_0$ acts as a reflection on $\mathfrak{a}^\ast$. As a consequence it admits a one dimensional eigenspace $E(1)$ with eigenvalue $1$ and a one dimensional eigenspace $E(-1)$ with eigenvalue $-1$. Thus the dimension of the $w_0$-invariant subspace of $\mathfrak{a}^\ast$ is one and the statement holds for $\ell=1$.

When $\ell=2$, it is immediate to see that there is no invariant elements in $\wedge^2\mathfrak{a}^\ast$ because any non trivial vector must be equivariant with respect to the sign of $w_0$. Thus the $w_0$-invariant cohomology vanishes in degree $2$ and the statement follows. 
\end{proof}

\begin{rem}
The proof of the previous lemma can be made more concrete via the choice of a basis for the space $\mathfrak{a}^\ast$. We denote by $e_i^\ast$ the linear functional defined on $\mathfrak{a}$ whose value is the $i$-th diagonal component of the matrix. A basis for $\mathfrak{a}^\ast$ is given by $\mathcal{B}=\{e_1^\ast - e_3^\ast, e_2^\ast \}$. 

The adjoint action of $w_0$ (whose expression is explicitely given below) on the basis $\mathcal{B}$ is given by 
$$
\textup{Ad}(w_0)(e_2^\ast)=e_2^\ast  , \ \ \textup{Ad}(w_0)(e_1^\ast - e_3^\ast)=-(e_1^\ast-e_3^\ast)  .
$$
We immediately see that the only vector fixed by $w_0$ is $e_2^\ast$. As a consequence 
$$
(\mathfrak{a}^\ast)^{w_0} \cong \langle e_2^\ast \rangle  .
$$

In a similar way, when $\ell=2$, an invariant element must be the exterior product of two elements of $\mathfrak{a}^\ast$ that are either both invariant or both equivariant. As a consequence there is no invariant vector with respect to the action of $w_0$ in degree $2$.
\end{rem}

\subsection{The $\pi_A$-projection for $\mathrm{SL}(3,\mathbb{K})$}\label{SL3R3} As we did for products of isometries of real hyperbolic spaces, here we need to give an explicit characterization of the projection map $\pi_A$. As before, we choose $A$ to be the maximal abelian subgroup of diagonal matrices with positive real entries and $P<G$ is the subgroup of upper triangular matrices with entries in $\mathbb{K}$. Furthermore we fix $K=\mathrm{SO}(3,\R)$, if $\mathbb{K}=\mathbb{R}$, or $K=\mathrm{SU}(3)$, if $\mathbb{K}=\mathbb{C}$. The unipotent radical of $P$ is the subgroup $N$ of triangular matrices having ones on the diagonal. 

Again by the Iwasawa decomposition, we can write any element $g \in G$ in a unique way as a product $ank$, where $a \in A, n \in N, k \in K$. Thanks to such decomposition we are allowed to define
$$
\pi_A:G \rightarrow A, \ \ g=ank \mapsto a .
$$

\begin{lem}\label{lem iwasawa sl3}
For any $g \in G$ consider the matrix $s:=gg^*$, where either $g^*=g^t$ if $\mathbb{K}=\mathbb{R}$ or $g^*=\overline{g}^t$ if $\mathbb{K}=\mathbb{C}$. If $(s_{ij})_{i,j=1,2,3}$ are the entries of $s$, the $A$-projection of $g$ is given by
$$
\pi_A(g)=\mathrm{diag}(a_{11},a_{22},a_{33}) ,
$$
where
\begin{align*}
a_{33}&=\sqrt{s_{33}}  , \\
a_{22}&=\sqrt{\frac{s_{22}s_{33}-|s_{23}|^2}{s_{33}}} , \\
a_{11}&=\frac{1}{a_{22}a_{33}}.
\end{align*}
\end{lem}

\begin{proof}
Let $X$ be the Riemannian symmetric space associated to $G$. For the convenience of the reader we remind that $X$ is the set of positive definite either symmetric matrices, when $\mathbb{K}=\mathbb{R}$, or Hermitian matrices, when $\mathbb{K}=\mathbb{C}$, with determinant equal to one. 

The natural projection map is given by 
$$
p:G \rightarrow X \ , \ \ p(g)=gg^*=s .
$$
Let $g=ank$ be the Iwasawa decomposition of the element $g$. Recall that $n$ is an upper triangular unipotent matrix with entries in $\mathbb{K}$ and $a$ is diagonal with positive real entries. We have that 
\begin{equation}\label{eq cholesky}
s=gg^*=(ank)(ank)^*=ankk^*n^*a^*=(an)(n^*a^*)  ,
\end{equation}
where we exploited the fact that $k$ is either orthogonal or unitary. One can verify algebraically that Equation \eqref{eq cholesky} is equivalent to the stament of the lemma. This concludes the proof. 
\end{proof}

\subsection{Contracting homotopies and differentials for $\mathrm{SL}(3,\mathbb{K})$} \label{SL3R4}

We start by fixing as a representative of the longest element $w_0$ in the Weyl group the following matrix
$$
w_0=\left(
\begin{array}{ccr}
0 & 0 & -1 \\
0 & 1 & 0 \\
1 & 0 & 0 
\end{array}
\right).
$$

Recall that given a configuration $[F_0,F_1,F_2]$ with triple ratio $\tau$, it is always possible to write a representative of such class as in Equation \eqref{Standard triple flags}. We fix such a choice of representative and we call it \emph{standard normalization}. A direct computation shows that a unipotent matrix fixing $F_0$ (the canonical flag) and sending $F_1$ to $F_2$ is given by 
$$
n_\tau:=\left(
\begin{array}{ccc}
1 & \frac{\tau+1}{\tau} & 1 \\
0 & 1 & 1\\
0 & 0 & 1\\
\end{array}
\right).
$$
In a similar way, a matrix sending the pair $(F_0,F_1)$ to the pair $(F_1,F_2)$ is given by 
$$
g_\tau:=\left(
\begin{array}{ccc}
0 & 0 & 1 \\
0 & 1 & 1\\
1 & \tau+1 & 1\\
\end{array}
\right).
$$

\begin{rem} 
Fixing the standard normalization is equivalent to say that a canonical representative for a configuration $[F_0,F_1,F_2]$ with triple ratio $\tau$ is given by $(F_{can},w_0F_{can},n_\tau w_0 F_{can})$, where $F_{can}$ is the canonical flag and $w_0$ represents the longest element in the Weyl group. 
\end{rem}

As for the case of products, here contracting homotopies for the cocomplexes $(C^{*,q}_K,d^\uparrow)$, for $q\geq 3$, can easily be defined as follows: 
$$\begin{array}{rccl}
h^{p,q}:&C^{p,q}_K&\longrightarrow & C^{p-1,q}_K\\
\end{array}
$$ 
with 
\begin{align*}
h^{p,q}(f)(g_0,\dots, g_{p-1})(&F_{can}, w_0F_{can}, n_\tau w_0 F_{can}, F_4,\dots, F_q):= \\
             f(e,g_0,\dots, g_{p-1})(&F_{can}, w_0F_{can}, n_\tau w_0 F_{can}, F_4,\dots, F_q).
\end{align*}

The previous definition gives back a measurable function on a dense subset of $G^p$: Indeed, although $f$ is only defined on a subset of full measure of $G^{p+1}$, the evaluation on $(e,g_0,\dots,g_{p-1})$ makes sense on a subset of full measure of $G^p$ by the $G$-invariance of $f$. Likewise, the evaluation on $(F_{can}, w_0F_{can}, n_\tau w_0 F_{can}, F_4,\dots, F_q)$ is allowed for almost all $(F_4,\dots,F_q)$ by transitivity and by our choice of normalization. The resulting cochain is clearly $G$-invariant by construction and it holds that
$$
h^{p+1,q} \circ d^\uparrow + d^\uparrow \circ h^{p,q}=\id  .
$$

Before computing explicitly the differential $d^\rightarrow:C^{p,2} \rightarrow C^{p,3}$, we want to point out that the lack a transitivity of the $\mathrm{SL}(3,\mathbb{K})$-action on triples of flags has important consequences on the realization of the space $C^{p,3}$. In fact, since the action is no more transitive on triples (as it were for products of isometries of real hyperbolic spaces), the space of orbits does not boil down to a point, but it is actually not trivial. In fact, we know that it is parametrized by the triple ratio. As a consequence 
$$
C^{p,3}=L^0(G^{p+1},L^0((G/P)^3))^G \cong L^0(G^{p+1} \times \mathbb{K}^{\ast\ast}) ,
$$
where the $\mathbb{K}^{\ast\ast}$-factor, for $\mathbb{K}^{\ast\ast}=\mathbb{K}^\ast \setminus \{-1\}$, is the contribution coming from the triple ratio. 

\begin{lem}\label{lem differential sl3}
The differential 
$$
d^\rightarrow: L^0(G^{p+1})^A \cong C^{p,2} \rightarrow L^0(G^{p+1} \times \mathbb{K}^{\ast\ast} ) \cong C^{p,3}  
$$
is given by 
$$
d^\rightarrow (\beta)(g_0,\ldots,g_p)(\tau)=(-1)^{p+1}[\beta(g_\tau^{-1}g_0,\ldots,g_\tau^{-1}g_p) -\beta(n_\tau^{-1}g_0,\ldots,n_\tau^{-1}g_p)+\beta(g_0,\ldots,g_p)].
$$
\end{lem}

\begin{proof}
The proof is analogous to the one of Lemma \ref{d 2 to 3}. Any cochain $\beta \in L^0(G^{p+1})^A$ determines a cochain $\overline{\beta} \in C^{p,2}=L^0(G^{p+1},L((G/P)^2))^G$ by setting 
$$
\overline{\beta}(g_0,\ldots,g_p)(h^{-1}F_{can},h^{-1}w_0F_{can}):=\beta(hg_0,\ldots,h_gp) .
$$
Fixing our choice of the standard normalization, we can see that the differential $d^\rightarrow (\beta) \in L^0(G^{p+1} \times \mathbb{K}^{\ast\ast})$ is given by 
$$
d^\rightarrow(\beta)(g_0,\ldots,g_p)(\tau)=d^\rightarrow \overline{\beta}(g_0,\ldots,g_p)(F_{can},w_0F_{can},n_\tau w_0 F_{can})  ,
$$
where the latter evaluation makes sense because of the $G$-invariance of $d^\rightarrow \overline{\beta}$.

By the definition of $d^\rightarrow$ we have that 
\begin{align*}
d^\rightarrow \overline{\beta}(g_0,\ldots,g_p)(F_{can},w_0F_{can},n_\tau w_0 F_{can})=&(-1)^{p+1}[\overline{\beta}(g_0,\ldots,g_p)(w_0 F_{can}, n_\tau w_0 F_{can}) \\
 &-\overline{\beta}(g_0,\ldots,g_p)(F_{can},n_\tau w_0 F_{can})\\
 &+\overline{\beta}(g_0,\ldots,g_p)(F_{can},w_0 F_{can})] . 
\end{align*} 

By the way we defined $g_\tau$ we know that
$$
(w_0 F_{can}, n_\tau w_0 F_{can})=(g_\tau F_{can}, g_\tau w_0 F_{can}) ,
$$
and thanks to the fact that $n_\tau F_{can}=F_{can}$, we can rewrite 
\begin{align*}
d^\rightarrow \overline{\beta}(g_0,\ldots,g_p)(F_{can},w_0F_{can},n_\tau w_0 F_{can})=&(-1)^{p+1}[\overline{\beta}(g_0,\ldots,g_p)(g_\tau F_{can}, g_\tau w_0 F_{can}) \\
 &-\overline{\beta}(g_0,\ldots,g_p)(n_\tau F_{can},n_\tau w_0 F_{can})\\
 &+\overline{\beta}(g_0,\ldots,g_p)(F_{can},w_0 F_{can})]  \\
=&(-1)^{p+1}[\beta(g_\tau^{-1}g_0,\ldots,g_\tau^{-1}g_p)-\beta(n_\tau^{-1}g_0,\ldots,n_\tau^{-1}g_p)\\
&+\beta(g_0,\ldots,g_p)],
\end{align*} 
and the statement is proved. 
\end{proof}

\subsection{Proof of Theorem \ref{Thm 1 to 2 for sl3}}\label{sec proof 1 2 sl3}\label{SL3R5}
Let $\alpha_\mathfrak{a} \in \mathfrak{a}^\ast$ be a $w_0$-invariant linear functional. We consider $\widetilde{\alpha}:A \rightarrow \R$ the inhomogeneous cocycle defined by
$$
\widetilde{\alpha}(a)=\alpha_{\mathfrak{a}}(\log a) ,
$$
and let $\alpha:A^2 \rightarrow \R$ be its homogenized variant
$$
\alpha(a_0,a_1):=\widetilde{\alpha}(a_0^{-1}a_1)  .
$$ 
We can consider the $G$-invariant extension of $\alpha$ to $G$ by precomposing with the projection $\pi_A$ of Lemma \ref{lem iwasawa sl3}, that is
$$
\alpha_G:G^2 \rightarrow \R \ , \ \alpha(g_0,g_1)=\alpha_G(\pi_A(g_0),\pi_A(g_1)) .
$$
It is clear that $\alpha_G$ is $P$-invariant, thus it lies in $C^{1,1}$, and its cohomology class in $H^2(C^{1,1},d^\uparrow) \cong H^1_m(P) \cong H^1_m(A) \cong \mathfrak{a}^\ast$ corresponds to $\alpha_{\mathfrak{a}}$. 

By Proposition \ref{Section} we have to follow $\overline{\alpha_G}$ through the maps

\begin{equation}\label{eq:diagram:chasing:sl3}
\xymatrix{
\overline{\alpha_G} \in C^{1,1} \ar[r] & d^\rightarrow \overline{\alpha_G}=d^\uparrow \beta \in C^{1,2} & \\
& \mathcal{H}^1(\overline{\alpha_G})=:\beta \in C^{0,2} \ar[u] \ar[r] & d^\rightarrow \beta=:\omega\in C^{0,3}. \\}
\end{equation}

\subsubsection*{Computation of $\beta$} By the definition of $\mathcal{H}^1$, we have
$$
\beta(g)=\mathcal{H}^1(\overline{\alpha_G})(g)=\alpha_G(w_0^{-1}\pi_A(g),w_0^{-1}g) .
$$

\subsubsection*{Computation of $\omega$} In this case the computation has only two steps. Thus $\omega(e)$ will give us back the desired cocycle. In virtue of Equation \eqref{Standard triple flags} we can fix $F_{can},w_0F_{can},n_\tau w_0F_{can}$ as triple of flags with triple ratio $\tau$. 
On such configuration we will evaluate $\omega(e)$. By Lemma \ref{lem differential sl3} we have that
\begin{align*}
\omega(e)(F_{can},w_0F_{can},n_\tau w_0 F_{can})&=d^\rightarrow \beta(e)(F_{can},w_0F_{can},n_\tau w_0 F_{can})  \\
									 &=-\beta(g_\tau^{-1})+\beta(n_\tau^{-1})-\beta(e) \\
					             	                &=-\alpha_G(w_0^{-1}\pi_A(g_\tau^{-1}),w_0^{-1}g_\tau^{-1})\\
								        &+\alpha_G(w_0^{-1}\pi_A(n_\tau^{-1}),w_0^{-1}n_\tau^{-1})  ,
\end{align*}
where we exploited the fact that $\beta(e)$ vanishes. 

By Lemma \ref{lem invariant SL3} we know that $(\mathfrak{a}^\ast)^{w_0}$ is one dimensional and generated by the functional $e_2^\ast$. This implies that we can write $\alpha_{\mathfrak{a}}$ as a multiple of $e_2^\ast$, say $\lambda e^\ast_2$. In an analogous way, the cocycle $\widetilde{\alpha}$ is a multiple of the logarithm of the second diagonal coordinate. 

Using Lemma \ref{lem iwasawa sl3} one can verify that the second diagonal coordinate of the $A$-projections are given by
\begin{align*}
(\pi_A(w_0^{-1}\pi_A(g_\tau^{-1})))_{22}&=1  , \ \ (\pi_A(w_0^{-1}g_\tau^{-1}))_{22}=\frac{\sqrt{3}}{\sqrt{2(|\tau|^2+\mathrm{Re}(\tau)+1)}}, \\
(\pi_A(w_0^{-1}\pi_A(n_\tau^{-1})))_{22}&=1   , \ \ (\pi_A(w_0^{-1}n_\tau^{-1}))_{22}=\frac{\sqrt{3}|\tau|}{\sqrt{2(|\tau|^2+\mathrm{Re}(\tau)+1)}}.
\end{align*}
As a consequence we get that
$$
\alpha_G(w_0^{-1}\pi_A(g_\tau^{-1}),w_0^{-1}g_\tau^{-1})=\lambda \log\left( \frac{\sqrt{3}}{\sqrt{2(|\tau|^2+\mathrm{Re}(\tau)+1)}} \right) , 
$$
and similarly for the other term we have
$$
\alpha_G(w_0^{-1}\pi_A(n_\tau^{-1}),w_0^{-1}n_\tau^{-1})=\lambda \log \left( \frac{\sqrt{3}|\tau|}{\sqrt{2(|\tau|^2+\mathrm{Re}(\tau)+1)}} \right) ,
$$

Summing everything up we obtain
$$
\omega(e)(F_{can},w_0F_{can},n_\tau w_0 F_{can})=\lambda \log |\tau|  ,
$$
and this concludes the proof. 

\subsection{Proof of Theorem \ref{Thm 1 to 3 for sl3}}\label{sec proof thm 1 3 sl3}\label{SL3R6}
We will use the same notation we used at the beginning of the previous section. Let $\alpha_\mathfrak{a}$ be a $w_0$-invariant linear functional defined on $\mathfrak{a}$, let $\widetilde{\alpha}$ be the associated inhomogeneous cocycle on $A$ and let $\alpha$ be its homogenization. By precomposing with the projection $\pi_A$ we obtain the cocycle $\alpha_G:G^2 \rightarrow \R$ defined on $G$. 

This time we view $\alpha_G$ as an $A$-invariant cocycle lying in $C^{1,2}$ and its cohomology class represents the element $\alpha_{\mathfrak{a}}$ in $H^2(C^{1,2},d^\uparrow) \cong H^1_m(A) \cong \mathfrak{a}^\ast$. 

By Lemma \ref{Injection} we have to follow $\overline{\overline{\alpha_G}}$ along the path 

\begin{equation}\label{eq:diagram:chasing:sl3:deg3}
\xymatrix{
\overline{\overline{\alpha_G}} \in C^{1,2} \ar[r] & d^\rightarrow \overline{\overline{\alpha_G}}=d^\uparrow \beta \in C^{1,3} & \\
& h^{0,3}(d^\rightarrow \overline{\overline{\alpha_G}})=:\beta \in C^{0,3} \ar[u] \ar[r] & d^\rightarrow \beta=:\omega\in C^{0,4}. \\}
\end{equation}

\subsubsection*{Computation of $\beta$} We start by computing the differential $d^\rightarrow \alpha_G$. By Lemma \ref{lem differential sl3} we have that 
\begin{align*}
d^\rightarrow \alpha_G(g_0,g_1)(\tau)&=\alpha_G(g_\tau^{-1}g_0,g_\tau^{-1}g_1)-\alpha_G(n_\tau^{-1}g_0,n_\tau^{-1}g_1)+\alpha_G(g_0,g_1)\\
&=\alpha_G(g_\tau^{-1}g_0,g_\tau^{-1}g_1) , 
\end{align*}
where we moved from the first line to the second one exploiting the $P$-invariance (and thus the $N$-invariance) of $\alpha_G$. 

To construct $\beta$ it is sufficient to apply to $d^\rightarrow \alpha$ the contracting homotopy $h^{0,3}$, that means
$$
\beta(g)(\tau)=d^\rightarrow \alpha_G(e,g)(\tau)=\alpha_G(g_\tau^{-1},g_\tau^{-1}g)  .
$$

\subsubsection*{Computation of $\omega$} 

Also in this case the computation has only two steps and $\omega(e)$ will give back the desired cocycle. Let $F_0,F_1,F_2,F_3$ be a $4$-tuple of flags in general position. We assume that their configuration class has coordinates 
$$
(z_{01},z_{10},z_{23},z_{32})=(a,b,c,d) .
$$ 
We denote by $\tau_{ijk}$ the triple ratio of the configuration $[F_i,F_j,F_k]$, where $i,j,k \in \{0,1,2,3\}$. Again by the transitivity of $G$, we can suppose that our flags are exactly the ones given by Equation \eqref{Standard four flags}. 

By the definition of the differential $d^\rightarrow$ we have that 
\begin{align*}
\omega(e)(F_0,F_1,F_2,F_3)&=d^\rightarrow(\beta)(e)(F_0,F_1,F_2,F_3) \\
&=-\beta(e)(F_1,F_2,F_3)+\beta(e)(F_0,F_2,F_3) \\
&-\beta(e)(F_0,F_1,F_3)+\beta(e)(F_0,F_1,F_2) .
\end{align*}
It is easy to see that the last term vanishes. For the other three summands we need to exploit $G$-invariance to compute the evaluation, since we know only how to evaluate $\beta(g)$, for some $g$, on a standard triple $(F_{can},w_0F_{can},n_\tau w_0 F_{can})$.  

As a result, we need to find  the matrices $g_{ijk} \in G$ satisfying
\begin{equation}\label{gijk}
g_{ijk}(F_{can},w_0F_{can},n_{\tau_{ijk}}w_0F_{can})=(F_i,F_j,F_k) ,
\end{equation}
where $(i,j,k)$ is either $(1,2,3),(0,2,3)$ or $(0,1,3)$. If we set $$r=a(1-b)  , \ s=1-ab-(\tau_{012}+1)r,$$ and $$u=(a-1)/\tau_{012} , \ v=ab-1-(1+\tau_{012})u,$$ 
we can define
$$
g'_{123}:=\left(
\begin{array}{ccc}
0 & 0 & ab\\
0 & r & ab \\
s & (1+\tau_{012})r & ab
\end{array} 
\right), \ \
g'_{023}:=\left(
\begin{array}{ccc}
v & (\tau_{012}+1)u  & 1\\
0 & \tau_{012}u & 1 \\
0 & 0 & 1
\end{array} 
\right), \ \ 
g'_{013}:=\left(
\begin{array}{ccc}
ab & 0 & 0\\
0 & a & 0\\
0 & 0 & 1
\end{array}
\right).
$$
Then the matrices  
$$
g_{ijk}=\frac{1}{\sqrt[3]{\det(g'_{ijk})}}g'_{ijk} 
$$
satisfy (\ref{gijk}).

Exploiting the matrices $g_{ijk}$ and the $G$-invariance, we can rewrite
\begin{align*}
\omega(e)(F_0,F_1,F_2,F_3)&=-\beta(g_{123}^{-1})(F_{can},w_0F_{can},n_{\tau_{123}}w_0F_{can})\\
&+\beta(g_{023}^{-1})(F_{can},w_0F_{can},n_{\tau_{023}}w_0F_{can}) \\
&-\beta(g_{013}^{-1})(F_{can},w_0F_{can},n_{\tau_{013}}w_0F_{can}) .
\end{align*}
The explicit computation of $\beta(g)$ we made in the previous subsection allows us to write the following equation 
\begin{align}\label{formula}
\omega(e)(F_0,F_1,F_2,F_3)&=-\alpha_G(g_{\tau_{123}}^{-1},g_{\tau_{123}}^{-1}g_{123}^{-1})+
\alpha_G(g_{\tau_{023}}^{-1},g_{\tau_{023}}^{-1}g_{023}^{-1})-\alpha_G(g_{\tau_{013}}^{-1},g_{\tau_{013}}^{-1}g_{013}^{-1})
\end{align}

Again by Lemma \ref{lem invariant SL3} the functional $\alpha_{\mathfrak{a}} \in \mathfrak{a}^\ast$ must be a multiple of $e_2^\ast$, say $\lambda e^\ast_2$. Analogously the cocycle $\widetilde{\alpha}$ is a multiple of the logarithm of the second diagonal coordinate. 

To evaluate $\alpha_G$, we thus need to compute the second coordinate of the $\pi_A$-projection of each of the six coordinates in the left hand side of (\ref{formula}). In view of Lemma \ref{lem iwasawa sl3} and the fact that these six coordinates are explicit it would be possible to compute those by hand. We however chose to exploit the  open source software \texttt{SageMath}, with which we obtained 


\begin{align*}
&\alpha_G(g_{\tau_{123}}^{-1},g_{\tau_{123}}^{-1}g_{123}^{-1})={\scriptstyle
\frac{\lambda}{6} \log 
\frac{
27|a|^4|a-1|^2|b|^2|b-1|^2|c-1|^2
}
{
\left| 8(|a|^2|b|^2+|a|^2\mathrm{Re}(c\overline{b})+|a|^2|c|^2-2|a|^2\mathrm{Re}(b)-|a|^2\mathrm{Re}(c)-\mathrm{Re}(ab\overline{c})-2|c|^2\mathrm{Re}(a)+|a|^2+\mathrm{Re}(c\overline{a})+|c|^2)^3 \right|
},
}
\\
&\alpha_G(g_{\tau_{023}}^{-1},g_{\tau_{023}}^{-1}g_{023}^{-1})={\scriptstyle
\frac{\lambda}{6} \log 
\frac{
27|a|^2|a-1|^2|b-1|^2|c|^4|c-1|^2
}
{
\left| 8(|a|^2|b|^2+|a|^2\mathrm{Re}(c\overline{b})+|a|^2|c|^2-2|a|^2\mathrm{Re}(b)-|a|^2\mathrm{Re}(c)-\mathrm{Re}(ab\overline{c})-2|c|^2\mathrm{Re}(a)+|a|^2+\mathrm{Re}(c\overline{a})+|c|^2)^3 
\right|
},
}
\\
&\alpha_G(g_{\tau_{013}}^{-1},g_{\tau_{013}}^{-1}g_{013}^{-1})=\frac{\lambda}{6} \log \left|
\frac{b}{a}
\right|^2.
\end{align*}

Putting everything together we obtain that  
$$
\omega(e)(F_0,F_1,F_2,F_3)=-\frac{2\lambda}{3} \log \left| \frac{b}{c} \right| ,
$$
and the statement is proved.

\subsection{A different proof of the injectivity of the comparison map (Theorem \ref{Injectivity comparison sl3}) for $\mathrm{SL}(3,\mathbb{K})$, when $\mathbb{K}=\mathbb{R},\mathbb{C}$}\label{SL3R7}

\begin{lem} \label{conj 7 for SL3} Let $\mathbb{K}$ be either the real or the complex field. Conjecture \ref{Conj injectivity boundary} is true for $G=\mathrm{SL}(3,\mathbb{K})$ in degree $3$, i.e. the comparison map 
$$c_G:H^3_{m,b}(G\curvearrowright G/P)\longrightarrow H^3_m(G\curvearrowright G/P)$$
is injective.
\end{lem}

\begin{proof} Take $f\in L^{\infty}((G/P)^4)$ a $G$-invariant cocycle representing a cohomology class in $H^3_{m,b}(G\curvearrowright G/P)$ lying in the kernel of the comparison map. We must have $f=\delta h$ for some not necessarily bounded $G$-invariant function $h\in L^0((G/P)^3)^G$. 

The $G$-invariance of both $f$ and $h$ implies that they both descend to two functions $\overline{f}$ and $\overline{h}$, respectively, on the configuration space. Additionally we have an equality $\overline{f}=\delta \overline{h}$ for almost every configuration class. More precisely, let $F_0,F_1,F_2,F_3$ be $4$ flags in general position. Every configuration $[F_i,F_j,F_k]$ has an associated triple ratio defined by
$$
x:=\tau(F_1,F_2,F_3)  , \ y:=\tau(F_0,F_2,F_3)  , \ z:=\tau(F_0,F_1,F_3)  , \ w:=\tau(F_0,F_1,F_2) .
$$

The equation $f=\delta h$ is equivalent to saying that $\overline{f}$ depends only on the triple ratios in the following way
\begin{equation}\label{eq cocycle triple ratio}
\overline{f}(x,y,z,w)=\overline{h}(x)-\overline{h}(y)+\overline{h}(z)-\overline{h}(w) ,
\end{equation}
for almost every $x,y,z,w \in \mathbb{K}^\ast\setminus \{-1\}$ satisfying the multiplicative cocycle condition $xz=yw$ (which follows by Equation \eqref{eq triple in terms cross}). By introducing the variable $a=z/y$, we can rewrite
$$
\overline{f}(x,y,z,w)=\overline{h}(x)-\overline{h}(y)+\overline{h}(z)-\overline{h}(ax)  .
$$
If we interchange the roles of $y$ and $z$, respectively, we obtain 
$$
\overline{f}(x,z,y,w)=\overline{h}(x)-\overline{h}(z)+\overline{h}(y)-\overline{h}(a^{-1}x)  , 
$$
and the essential boundedness of $\overline{f}$ implies that there exists some $C>0$ such that
\begin{equation}\label{eq quasi morphism}
\left| 2\overline{h}(x)-\overline{h}(ax)-\overline{h}(a^{-1}x) \right| < C  ,
\end{equation}
for almost every $a,x \in \R^\ast$. 

Notice that Equation \eqref{eq quasi morphism} can be rewritten as
\begin{equation}\label{eq quasi morphism 2}
\left | 2\overline{h}(\sqrt{uv})-\overline{h}(u)-\overline{h}(v) \right| < C,
\end{equation}
where to move from Equation \eqref{eq quasi morphism} to Equation \eqref{eq quasi morphism 2} we used the change of variables $u=ax$ and $v=a^{-1}x$.

Our goal is to prove that $\overline{h}$ is at a bounded distance from a homomorphism. To show that, we will fix $\mathbb{K}=\mathbb{C}$, since the real case is simpler than the complex one. Unfortunately the function $\overline{h}$ is only measurable, thus we cannot evaluate it on specific values. To overcome such difficulty, we fix $\varepsilon>0$ and define the function
$$
H_\varepsilon: \mathbb{C}^\ast \rightarrow \R  , \ \ H_\varepsilon(x):=\frac{1}{4\pi\varepsilon}\int_{-\varepsilon}^\varepsilon \int_{-\pi}^{\pi} \overline{h}(2^te^{i \theta} x)dtd\theta. 
$$

The function $H_\varepsilon$ is obtained starting from $\overline{h}$ and averaging it on a small annulus centered in the origin and containing $x \in \mathbb{C}^\ast$ (see Figure \ref{figure annulus}). 

\begin{figure}[!h]
	\centering
		\includegraphics[width=7cm]{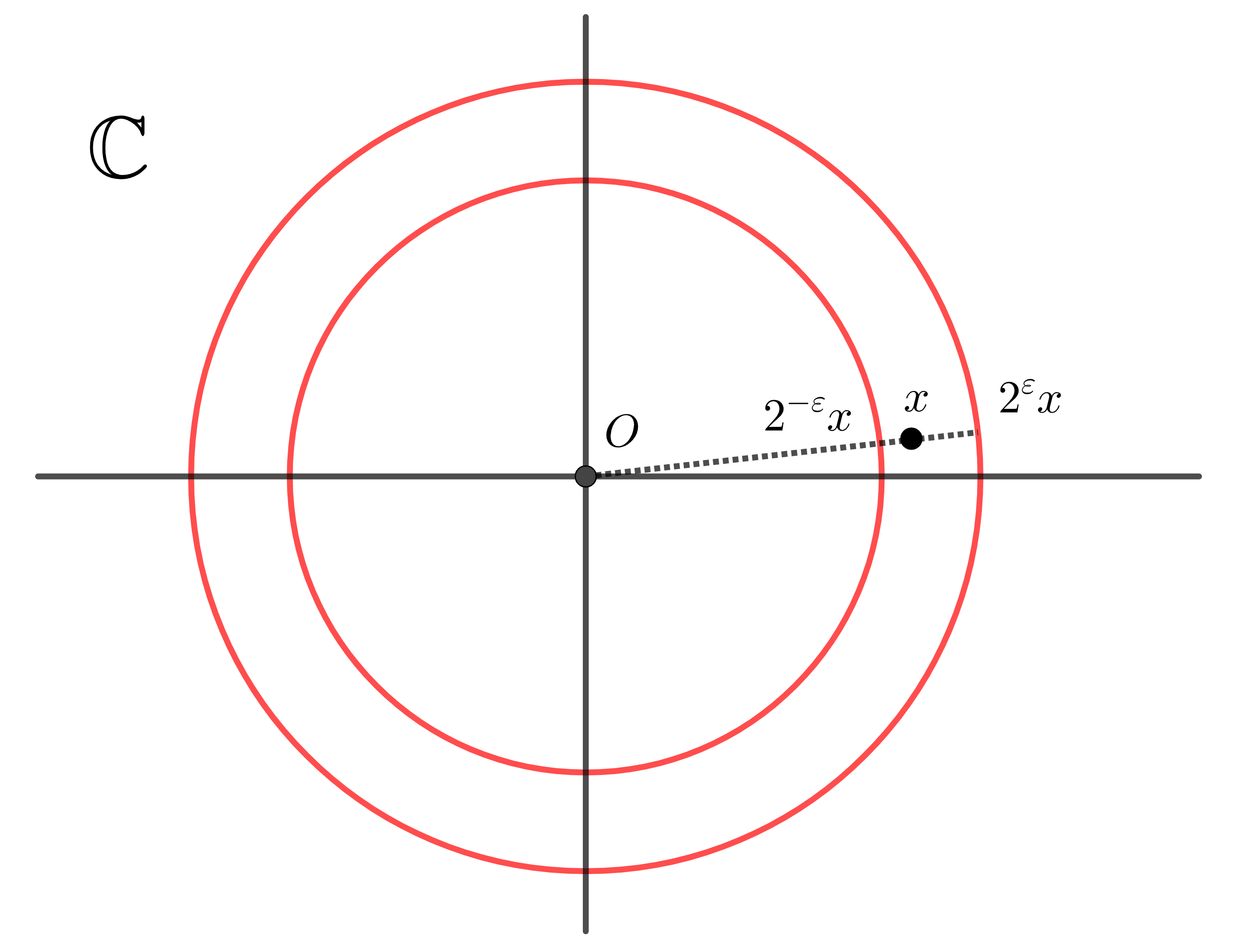}
	\caption{Annulus centered in the origin and containing $x \in \mathbb{C}^\ast$}\label{figure annulus}
\end{figure}

To prove that the integral defining $H_\varepsilon$ converges so that $H_\varepsilon$ is actually well-defined, we are going to prove that 
$$
|H_\varepsilon(x) - \overline{h}(x) |  \  
$$
is bounded for almost every $x \in \mathbb{C}^\ast$. To avoid a heavy notation, we are going to write $\sim$ to say that two quantities are at a bounded distance.

It holds that 
\begin{align*}
H_\varepsilon(x)&=\frac{1}{4\pi\varepsilon}\int_{-\varepsilon }^\varepsilon \int_{-\pi}^{\pi} \overline{h}(2^te^{i\theta}x)dtd\theta\\
&=\frac{1}{4 \pi \varepsilon}\int_{-\varepsilon}^{\varepsilon} \int_0^{\pi} (\overline{h}(2^te^{i\theta}x)+\overline{h}(2^te^{-i\theta}x))dtd\theta \\
&\sim \frac{1}{4\pi\varepsilon} \int_{-\varepsilon}^{\varepsilon}\int_{0}^{\pi} 2\overline{h}(2^tx)dtd\theta=\frac{1}{2\varepsilon}\int_{-\varepsilon}^{\varepsilon} \overline{h}(2^tx)dt\\
&=\frac{1}{2\varepsilon}\int_0^{\varepsilon} (\overline{h}(2^tx)+\overline{h}(2^{-t}x))dt\\
&\sim \frac{1}{2\varepsilon} \int_0^{\varepsilon} 2\overline{h}(x) dt=\overline{h}(x).
\end{align*}
We moved from the second line to the third one and from the fourth line to the fifth one using twice Equation \eqref{eq quasi morphism}. As a result, $H_\varepsilon$ is well-defined, continuous and for almost every $x \in \mathbb{C}^\ast$ is at bounded distance from $\overline{h}$. Since the same proof works also for $\mathbb{K}=\mathbb{R}$ (without actually integrating along circles), we will move back to consider the case of a general $\mathbb{K}$. 

For what we have shown so far, the function $H_\varepsilon$ must satisfy Equation \eqref{eq quasi morphism}, that is
\begin{equation}\label{eq new quasi morphism}
\left| 2H_\varepsilon(x)-H_\varepsilon(ax)-H_\varepsilon(a^{-1}x) \right|<C',
\end{equation}
for some $C'$ positive. Equivalently we must have that Equation \eqref{eq quasi morphism 2} holds, namely
\begin{equation}\label{eq new quasi morphism 2}
\left| 2H_\varepsilon(\sqrt{uv})-H_\varepsilon(u)-H_\varepsilon(v) \right|<C' .
\end{equation}

The continuity of $H_\varepsilon$ allows us to evaluate $H_{\varepsilon}$ on $x=a$, and Equation \eqref{eq new quasi morphism} implies that the quantity 
\begin{equation}\label{eq quasi morphism square}
|2H_\varepsilon(x)-H_\varepsilon(x^2)|  
\end{equation}
is uniformly bounded with respect to $x \in \mathbb{K}^\ast$. As a consequence, for almost every $u,v \in \mathbb{K}^\ast$, we must have
$$
H_\varepsilon(u)+H_\varepsilon(v) \sim 2 H_\varepsilon(\sqrt{uv}) \sim H_\varepsilon(uv) \ , 
$$
where the left-hand side is due to Equation \eqref{eq new quasi morphism 2} and the right-hand side is due to Equation \eqref{eq quasi morphism square}. As a consequence, $H_\varepsilon$ is a continuous quasi-morphism on $\mathbb{K}^\ast$. By the characterization of quasi-morphisms on $\mathbb{K}^\ast$, there must exist some $\lambda \in \mathbb{R}^\ast$ such that $H_\varepsilon$ is at bounded distance from the function $\lambda \log|x|$. The fact that $H_\varepsilon$ is at bounded distance from $\overline{h}$ implies that 
$$
\overline{\varphi}:\mathbb{K}^\ast \rightarrow \R \ , \ \ \ \overline{\varphi}(x):=\overline{h}(x)-\lambda \log|x| 
$$
is an essentially bounded function. Additionally it holds
$$
\overline{f}=\delta \overline{h}=\delta \overline{\varphi} .
$$
This proves that the class determined by $\overline{f}$ was already trivial in $H^3_{m,b}(G \curvearrowright G/P)$, and the statement is proved. 
\end{proof}

\begin{prop}\label{prop intersection trivial sl3}\label{Prop log b/c unbounded}
The function 
$$
\alpha:(F_0,F_1,F_2,F_3) \mapsto \log \left|  \frac{z_{10}}{z_{23}} \right|
$$
defines a non-trivial cocycle and it has no bounded representative in its cohomology class. 
\end{prop}

\begin{proof}
The non-triviality follows by the proof of Theorem \ref{Thm 1 to 3 for sl3}, since it is a generator of $NH^3_m(G \curvearrowright G/P)$. 

We have to prove that there is no function $f \in L^\infty((G/P)^4)$ and no function $h \in L^0((G/P)^3)$ such that
\begin{equation}\label{eq boundedness log}
\alpha(F_0,F_1,F_2,F_3)=f(F_0,F_1,F_2,F_3)+\delta h(F_0,F_1,F_2,F_3)  ,
\end{equation}
for almost every $F_0,\ldots,F_3 \in G/P$. We suppose that the configuration class $[F_0,F_1,F_2,F_3]$ has coordinates $(a,b,c,d)$ and,  as in the proof of the previous lemma, we set
$$
x:=\tau(F_1,F_2,F_3)  , \ y:=\tau(F_0,F_2,F_3)  , \ z:=\tau(F_0,F_1,F_3)  , \ w:=\tau(F_0,F_1,F_2)  ,
$$
where $xz=yw$ by Equation \eqref{eq triple in terms cross}. In this way we can rewrite Equation \eqref{eq boundedness log} as follows
\begin{equation}\label{eq configuration spaces unbounded}
\log|b/c|=\overline{f}(a,b,c,d)+\delta \overline{h}(x,y,z)  ,
\end{equation}
where $\overline{f}$ and $\overline{h}$ are the functions induced by $f$ and $h$ on the configuration spaces, respectively. Notice that we used the relation $xz=yw$ to get rid of the variable $w$ in the function $\overline{h}$. 

The key point now is that we can exploit Equation \eqref{eq triple in terms cross} to write the cross ratios $a,c,d$ in terms of the triple ratios $x,y,z$. More precisely, Equation \eqref{eq triple in terms cross} implies that 
\begin{align}\label{eq a c d in terms triple}
a&=\frac{y ((b-1)xz+bx+1)}{y(1 + b x) - (b-1) x z}, \\
c&=\frac{(b-1)xz+bx+1}{1 + y} , \nonumber \\
d&=\frac{b x (1 + y)}{y(1 + b x) - (b-1) x z} \nonumber  ,
\end{align}
where each equation has to be intended in its natural domain of definition. Additionally $a,c,d$ depend continuously on the triple ratios $x,y,z$ in such domain. If we fix a small neighborhood $U$ of positive measure of the point $(x,y,z)=(1,2,4)$, we have that the right-hand side of Equation \eqref{eq configuration spaces unbounded} is essentially bounded in $U$. On the other hand the left-hand side has the form 
$$
\log \left| \frac{b}{c(b,x,y,z)} \right|  ,
$$
and, for our choice of $x,y,z$, a continuity argument implies that the function lies in a neighborhood of the function
$$
\varphi(b):=\log \left| \frac{b}{c(b,1,2,4)} \right|=\log \left| \frac{3b}{5b-3} \right|  ,
$$
and the latter is clearly unbounded. This leads to a contradiction, hence there are no functions $f$ and $h$ satisfying Equation \eqref{eq boundedness log}. 
\end{proof}

We claim that Proposition \ref{prop intersection trivial sl3} proves Conjecture \ref{Conj Gap}. In fact, by Theorem \ref{Thm 1 to 3 for sl3} we know that any cohomology class in $NH^3_m(G \curvearrowright G/P)$ admits a representative which is a multiple of the function of Proposition \ref{prop intersection trivial sl3}. As a consequence any cohomology class in $NH^3_m(G \curvearrowright G/P)$ cannot admit an essentially bounded representative or, equivalently, the intersection of $NH^3_m(G \curvearrowright G/P)$ with the comparison map $H^3_{m,b}(G \curvearrowright G/P) \rightarrow H^3_m(G \curvearrowright G/P)$ is trivial, as claimed. Together with Lemma \ref{conj 7 for SL3}, this gives  another proof of Theorem \ref{Injectivity comparison sl3}.

\bibliographystyle{alpha}

\end{document}